\documentclass{article}

\usepackage[english]{babel}
\usepackage{graphicx}
\usepackage{subcaption}
\usepackage{amsmath}
\usepackage{amsthm}
\usepackage{amssymb}
\usepackage{mathrsfs}
\newtheorem{theorem}{Theorem}[section]
\newtheorem{lemma}[theorem]{Lemma}
\newtheorem{proposition}{Proposition}[section]
\newtheorem{corollary}[theorem]{Corollary}
\theoremstyle{definition}
\newtheorem{definition}[theorem]{Definition}

\theoremstyle{definition}           
\newtheorem{example}{Example}[section] 
\theoremstyle{remark}
\newtheorem{remark}[theorem]{Remark}

\usepackage{float}
\usepackage{enumitem}
\usepackage{accents}

\newcommand{\lPi}{\rotatebox[origin=c]{180}{$\Pi$}}

\usepackage{natbib}
\usepackage[letterpaper,top=2cm,bottom=2cm,left=3cm,right=3cm,marginparwidth=1.75cm]{geometry}

\usepackage{amsmath}
\usepackage{graphicx}
\usepackage[colorlinks=true, allcolors=blue]{hyperref}

\title{Efficiency of Valid Inferential Models: Choquet-risk Optimal Possibility Measures, and Direct Comparisons}
\author{
Max Raner\thanks{Department of Mathematics, Uppsala University, Uppsala, Sweden. 
Email: \texttt{max.raner@math.uu.se}.}
}

\begin{document}
\maketitle

\begin{abstract}

Valid possibilistic inferential models provide exact finite-sample calibration, but validity alone does not determine which of several valid procedures gives the most informative inference. This paper proposes Choquet risk as a decision-theoretic criterion for comparing valid possibility measures in finite samples. Given a non-negative penalty functional, Choquet loss is defined as the Choquet integral of that penalty with respect to the data-dependent possibility measure, and Choquet risk as its sampling expectation. A key reduction expresses this risk through the nested $\alpha$-cuts of the contour (its superlevel-sets), thereby linking procedure-level efficiency to the expected performance of calibrated confidence sets. For concentration penalties, the criterion reduces to integrated expected set size, equivalently expected contour volume, so levelwise optimal confidence sets induce Choquet-risk optimal valid contours.

The framework is developed along two classical routes to optimality. First, a possibilistic notion of unbiasedness is introduced and shown, under validity, to coincide with unbiasedness of the induced confidence sets and tests, allowing UMPU and most-accurate-unbiased results to be transferred to valid contours. Second, an equivariant minimax theory is developed, including a Gaussian-location result in which the Gaussian possibility contour is Choquet-risk minimax for radial distance-to-truth losses. The construction also extends confidence risk from additive confidence distributions to non-additive calibrated IM output, with Choquet loss acting as a least-favourable confidence loss. Finally, the paper clarifies the penalty-dependence of efficiency comparisons and motivates invariant size criteria and divergence-based intrinsic losses connected locally to Fisher--Rao geometry.

\end{abstract}



\section{Introduction}
In Efron's terminology, the ``holy grail of statistical theory'' is the longstanding aim of producing posterior-like uncertainty summaries without relying on a subjective prior \citep{efron2010future}. To this end, a variety of proposals have been put forth, including generalized fiducial inference \citep{Hannig2016}, confidence distributions \citep{xie2013, schweder_hjort_2016}, objective-Bayes procedures \citep{Jeffreys46, BergerObjective}, and valid inferential models (IMs) \citep{Martin2013, Martin2025}. Different proposals pursue this aim with different tradeoffs. At a high level one can think of four desiderata—four “inductive dials”—that any inferential procedure implicitly tunes: \textbf{calibration} (frequentist reliability), \textbf{coherence} (internal consistency of the uncertainty calculus), \textbf{efficiency} (informativeness subject to calibration), and \textbf{flexibility} (scope and ease of use across problems). The key point for the present paper is that efficiency is only meaningful relative to calibration. Without a reliability constraint, “sharp” inference can always be obtained by over-confident output. Conversely, calibrated inference is only meaningful insofar as it is efficient. Once calibration is in place then, the main difficulty lies in how to capture efficiency.

Valid inferential models \citep{Martin2013} — ``one of the original statistical innovations of the 2010s" \citep{demyst} — can be read as turning the calibration dial up to its maximum by enforcing assertion-wise frequentist guarantees. In the possibilistic formulation, the inferential output is a data-dependent possibility measure $\Pi_x$ (equivalently, a contour $\pi_x$; we defer the formal definitions to the background section). Calibration alone, however, does not determine a unique procedure: in a given problem there are typically many distinct valid IM constructions available to choose from, ranging from nearly vacuous to highly informative. Without a meaningful way to measure efficiency, we then have little guidance for choosing among them. This raises the central question that motivates this paper: \textbf{how should we define—and compare—the finite-sample efficiency of valid possibility measures?}

Efficiency has been an explicit but elusive theme in the IM literature. Empirically, valid possibilistic IMs often produce uncertainty summaries that are competitive with familiar default-prior Bayes or classical frequentist solutions, but until recently this was supported mainly by examples. The large-sample picture is now much clearer. A possibilistic Bernstein–von Mises (BvM) theorem \citep{Martin2025BvM} shows that, under Le Cam–type regularity, centered and scaled IM contours can be uniformly approximated by a Gaussian possibility contour corresponding to the limiting likelihood experiment and Cramér-Rao efficiency. In particular, there is no asymptotic efficiency loss from imposing exact finite-sample validity and imprecision. What remains missing is a general finite-sample notion of efficiency for valid possibility measures that (i) compares distinct valid IM constructions on a common scale, (ii) recovers classical confidence-set optimality when specialized, (iii) aligns with the large-sample picture asymptotically, and (iv) supports direct numerical comparisons when analytic optimality arguments are unavailable.

This paper's key contributions are as follows. This paper proposes a decision-theoretic answer that evaluates the procedures $\Pi_x$ directly. We introduce a family of risk-functionals for possibility measures built from Choquet integrals, the natural notion of expectation for non-additive uncertainty. Given a non-negative penalty functional $\Gamma$ that quantifies either accuracy (distance-to-truth) or concentration (tightness of plausibility regions), we define the \emph{Choquet loss} as the upper expectation of $\Gamma$ under $\Pi_x$, and the corresponding \emph{Choquet risk} as the sampling expectation of the data-dependent loss. A key feature is that Choquet risk decomposes as an integral of expected ``worst-case penalties''over a nested family of confidence sets. This reduction lets us lift classical optimality results for confidence sets—e.g., most accurate unbiased intervals and minimax equivariant regions—into a unified efficiency theory for valid possibility measures. When analytic optimization is unavailable or infeasible, the framework also enables direct, procedure-level comparisons (e.g., by simulation).

The approach is inspired by confidence loss and confidence risk (introduced to capture the efficiency of confidence distributions; see e.g. \cite[Ch.~5.3]{schweder_hjort_2016}), but here adapted to an imprecise probability setting via Choquet integration. Moreover, we make explicit connections between Choquet-risk and confidence risk.

Classical optimality results typically require additional structure. Unbiasedness can yield global best procedures within limited classes, while equivariance can replace unattainable pointwise optimality with a minimax principle and provide guidance in broader classes. Accordingly, after developing the general Choquet-risk framework, we study two complementary routes to optimality. First, we introduce a possibilistic notion of unbiasedness and show that, under validity, it coincides with classical unbiasedness of the induced confidence sets and tests. This permits a direct translation of classical UMPU and most-accurate unbiased confidence-set theory into concentration-optimal valid contours. Second, we develop an equivariant/minimax theory for valid contours: equivariance of a contour is equivalent to equivariance of all its $\alpha$-cuts, a Hunt--Stein-type reduction justifies restricting minimax comparisons to equivariant procedures, and in the Gaussian location experiment the standard Gaussian z-test possibility contour is shown to be Choquet-risk minimax for radial losses. Whilst showing the z-test is optimal in yet another way might not seem very novel, the kind of radial loss optimality it demonstrates is novel. Moreover, the proof of the result hints at an interesting connection between convex geometry and this new decision theory framework of Choquet risk.

The paper also clarifies how Choquet-risk comparisons depend on the chosen
parametrization and on the penalty used to evaluate a contour. Since risk is defined
relative to a penalty, a contour that is efficient for squared error in \(\theta\) need
not be efficient for squared error in \(\phi=g(\theta)\). This is not a failure of the
risk framework, but a change of the decision problem. If the original penalty is
transported along the reparameterization, the corresponding risk comparison is
unchanged. The more substantive question is therefore how to choose penalties that
are not tied to arbitrary coordinates. This motivates model-based choices such as
Haar-size concentration criteria and divergence-based accuracy losses. In regular
models, the latter reduce locally to Fisher--Rao quadratic loss, thereby connecting
the finite-sample theory to the Gaussian benchmark suggested by the possibilistic
Bernstein--von Mises theorem. This connection is used here as an asymptotic
roadmap, not as a completed local asymptotic minimax theorem.

The rest of the paper is organized as follows. Section \ref{sec:background} reviews the background on possibility measures, credal sets, Choquet integration, validity, standard IM constructions, and existing notions of efficiency, including the large-sample Bernstein--von Mises picture and the limitations of current finite-sample comparisons. Section~\ref{sec:theory} develops the theory. Section \ref{sec:choclossandrisk} defines Choquet loss and risk and proves the \(\alpha\)-cut decomposition. Section \ref{sec:optRisk} formulates Choquet-risk optimality and shows how levelwise confidence-set optimality lifts to possibility-contour optimality. Section \ref{sec:confrisk} relates Choquet risk to confidence risk. Section \ref{sec:pos-unb} develops possibilistic unbiasedness and its connection to classical unbiased tests and confidence sets. Section \ref{sec:equi} develops the equivariant/minimax theory, including the Gaussian-location benchmark and the transitive-model construction. Section \ref{sec:directComp} turns to direct simulation-based comparisons of concrete procedures. Section \ref{sec:reparinvariance} discusses reparametrisation, invariant size criteria, intrinsic divergence losses, and the asymptotic Fisher--Rao outlook. Section \ref{sec:conclusion} concludes, and the appendix contains the technical symmetrization and Gaussian-geometry arguments used in the minimax results.


\section{Background: possibility measures and validity}\label{sec:background}

\subsection{Possibility measures, contours, and credal sets}\label{sec:possibility}
In this paper we focus on a particularly simple and statistically natural imprecise probability formalism:
\emph{possibility measures} (see e.g. \cite{dubois_prade_1988}, or \cite{Dubois2006PossibilityStats}); they can also be seen as \emph{consonant belief functions} (e.g. \cite{shafer2020mathematical}), and have close ties to fuzzy set theory (e.g. \cite{Zadeh1978Possibility}). These are expressive enough to support assertion-wise calibration (validity),
yet structured enough to yield tractable representations (contours, nested credible sets) and a clean
notion of expectation (Choquet/upper expectation) that we use later to define risk. Instead of additive, possibility measures are \emph{maxitive} (i.e., $\Pi(A\cup B)=\max{\{\Pi(A),\Pi(B)\}}$), and therefore encode a different type of information. They express what is \emph{not ruled out}—an emphasis in alignment with the logic of classical statistical reasoning, where the central concern is what cannot be excluded.

A (normalized) possibility measure on $(\Theta,\mathcal{B}(\Theta))$ is a set function
$\Pi:\mathcal{B}(\Theta)\to[0,1]$ satisfying:
\begin{enumerate}
\item \textbf{Normality:} $\Pi(\varnothing)=0$ and $\Pi(\Theta)=1$.
\item \textbf{Maxitivity:} for any countable collection $\{A_i\}$,
\[
  \Pi\Big(\bigcup_i A_i\Big)=\sup_i \Pi(A_i).
\]
\end{enumerate}
Maxitivity replaces additivity: where probability assigns mass by \emph{integration},
possibility assigns plausibility by \emph{optimization}. Concretely, $\Pi$ is determined by a
\emph{possibility contour} $\pi:\Theta\to[0,1]$ satisfying $\sup_{\theta\in\Theta}\pi(\theta)=1$, via
\begin{equation}\label{eq:poss-from-contour}
  \Pi(A)=\sup_{\theta\in A}\pi(\theta),\qquad A\in\mathcal{B}(\Theta).
\end{equation}
Intuitively, $\pi(\theta)$ measures how plausible a specific parameter value $\theta$ is, while
$\Pi(A)$ measures how plausible an assertion $A$ is, in the sense that $A$ is ruled out
only if \emph{all} its points are implausible. Though $\pi(\theta)=\Pi(\{\theta\})$, so the contour of a possibility measure is just the surface of singleton assertion possibility assignments.

The dual of $\Pi$ is the \emph{necessity measure} $\lPi$, defined by conjugacy:
\[
  \lPi(A)=1-\Pi(A^c),\qquad A\in\mathcal{B}(\Theta),
\]
which can be interpreted as a degree of support \emph{in favor} of $A$ (while $\Pi(A)$ expresses
the extent to which $A$ is \emph{not ruled out}).

A convenient geometric representation is given by \emph{$\alpha$-cuts} (superlevel sets) of the contour:
\begin{equation}\label{eq:alpha-cuts}
  A_\alpha = A_\alpha(\pi) := \{\theta \in \Theta : \pi(\theta) \ge \alpha\},
\qquad \alpha \in [0,1],
\end{equation}
These sets are nested: if $\alpha_1<\alpha_2$, then $A_{\alpha_2}\subseteq A_{\alpha_1}$.
The set $A_1=\{\theta:\pi(\theta)=1\}$ is the \emph{core} of the contour: any set containing the core
has possibility $1$, and any set that fails to contain the core has necessity $0$. And the contour can be recovered from its $\alpha$-cuts via
\begin{equation}\label{eq:contcutrep}
    \pi(\theta)=\sup\{\alpha\in[0,1]:\theta\in A_\alpha\}.
\end{equation}

\subsubsection*{Credal set interpretation}

Possibility measures are coherent upper probabilities and therefore admit a standard
\emph{credal set} interpretation. The credal set associated with $\Pi$ is
\begin{equation}\label{eq:credal-set}
  \mathcal{C}(\Pi)=\{Q \in \mathrm{prob}(\Theta): Q(A)\le \Pi(A)\ \ \forall A\in\mathcal{B}(\Theta)\},
\end{equation}
where $\mathrm{prob}(\Theta)$ denotes the set of (countably additive) probability measures on
$(\Theta,\mathcal{B}(\Theta))$. In words, $\mathcal{C}(\Pi)$ is the set of ordinary probabilities that
are dominated by $\Pi$. Moreover, $\Pi$ can be recovered as the upper envelope of its credal set:
$\Pi(A)=\sup\{Q(A):Q\in\mathcal{C}(\Pi)\}$.

A key advantage of possibility measures is that $\mathcal{C}(\Pi)$ has a simple characterization in terms of
$\alpha$-cuts (see \citet{dubois2004probability}, and \citet{couso2001necessity}).  A probability $Q$ belongs to $\mathcal{C}(\Pi)$ if and only if it assigns conservative mass to
all $\alpha$-cuts, i.e.
\begin{equation}\label{eq:credal-alpha}
  Q\in\mathcal{C}(\Pi)\quad \Longleftrightarrow\quad Q(A_\alpha)\ge 1-\alpha,\ \ \forall \alpha\in[0,1].
\end{equation}
Equivalently, if $\Theta^\star\sim Q$, then $\pi(\Theta^\star)$ is stochastically no smaller than $\mathrm{Unif}(0,1)$.

\subsubsection*{Upper expectations and the Choquet integral}\label{sec:upper-expectation}

If $\Pi$ is viewed as an upper probability, then a natural notion of “expectation” for a measurable
$f:\Theta\to[0,\infty)$ is the \emph{upper expectation}
\begin{equation}\label{eq:upper-exp}
  \overline{E}_\Pi[f] \;=\; \sup_{Q\in\mathcal{C}(\Pi)} E_Q[f].
\end{equation}
This upper expectation coincides with the Choquet integral of $f$ with respect to $\Pi$ \cite{TroffaesDeCooman2014}.
For possibility measures, it admits a particularly transparent $\alpha$-cut representation (see \cite{TroffaesDeCooman2014} results 7.14 and 15.42): for every nonnegative $\mathcal B(\Theta)$-measurable $f$ such that the Choquet integral exists,
\begin{equation}\label{eq:choquet-poss}
  \overline{E}_\Pi[f]
  \;=\; \int f d\Pi \;=\;
  \int_0^1 \sup_{\theta\in A_\alpha} f(\theta)\, d\alpha.
\end{equation}
Formula \eqref{eq:choquet-poss} is the form used later when we define loss and risk for
data-dependent possibility measures: it expresses the upper expectation as an average over
plausibility levels $\alpha$, of the worst-case value of $f$ over the corresponding $\alpha$-cut.

\subsection{Validity of data-dependent possibility measures}\label{sec:validity}

An \emph{inferential model} (IM) in the present paper is a mapping
\[
  x \longmapsto (\lPi_x,\Pi_x),
\]
where $\Pi_x$ is a possibility measure on $(\Theta,\mathcal{B}(\Theta))$ with contour $\pi_x$ and
$\lPi_x(A)=1-\Pi_x(A^c)$ is its necessity measure. We focus primarily on $\Pi_x$ (or equivalently $\pi_x$),
since $\lPi_x$ is determined by conjugacy.

The defining reliability property of (possibilistic) IMs is \emph{validity}, an assertion-wise frequentist
calibration condition \citep{Martin2013}.
\begin{definition}
    We say that the inferential model $x\mapsto\Pi_x$ is \emph{valid} if 
\begin{equation}\label{eq:validity-set}
  \sup_{\theta\in H} P_\theta\{\Pi_X(H) < \alpha\}\ \le\ \alpha, \quad \text{for all $H\in\mathcal{B}(\Theta)$, $\alpha\in[0,1]$}.
\end{equation}
\end{definition}
In words: when $H$ is true, the event that the IM assigns \emph{small} possibility to $H$
is controllably rare, uniformly over the true parameter values in $H$. By conjugacy, if $H$ is false, validity can equivalently be characterised as the event that the IM assigns \emph{high} necessity (support) to $H$ is controllably rare.

Since $\Pi_x(H)=\sup_{\vartheta\in H}\pi_x(\vartheta)$, validity can equivalently be expressed at the
singleton level:
\begin{equation}\label{eq:validity-point}
  P_\theta\{\pi_X(\theta) < \alpha\}\ \le\ \alpha,\qquad \forall \theta\in\Theta,\ \forall \alpha\in[0,1].
\end{equation}
Thus, for each fixed \(\theta\), the random variable \(\pi_X(\theta)\) is stochastically no smaller than \(\mathrm{Unif}(0,1)\) under \(P_\theta\). 
This version is often called \emph{strong} validity, as it is the calibration criteria ensuring credible sets are confidence sets. In this setting (under vacuous prior information), definitions \eqref{eq:validity-set} and \eqref{eq:validity-point} are equivalent, but in settings with genuine (partial) prior information, the latter is a \emph{stronger} calibration criteria (see \cite{Martin2023b}). This is the same calibration property enjoyed by possibly conservative \(p\)-values, so \(\pi_x(\theta)\) may be read, at the singleton level, as a calibrated measure of evidence against the point assertion \(\{\theta\}\). 
The point, however, is not merely that possibility contours behave like collections of \(p\)-values. Rather, the \(p\)-value interpretation is embedded in a richer possibilistic uncertainty calculus: the singleton contour \(\theta \mapsto \pi_x(\theta)\) determines the assertion-wise upper probability \(\Pi_x(H)=\sup_{\theta\in H}\pi_x(\theta)\), its dual necessity measure, and the nested family of calibrated plausibility regions.

\begin{remark}
    We use a strict ``$<$'' in \eqref{eq:validity-set}--\eqref{eq:validity-point} to align the superlevel sets
$\{\pi_x\ge \alpha\}$ with closed confidence regions (important in discrete problems). In continuous
models, the distinction between ``$<$'' and ``$\le$'' is typically immaterial.
\end{remark}

Validity implies that $\alpha$-cuts of the contour are calibrated confidence regions. Define the
data-dependent $\alpha$-cut
\begin{equation}\label{eq:valid-alpha-cut}
  A_\alpha(X)=\{\theta\in\Theta:\pi_X(\theta)\ge \alpha\},\qquad \alpha\in[0,1].
\end{equation}
Then \eqref{eq:validity-point} yields the exact finite-sample coverage statement
\begin{equation}\label{eq:coverage}
  P_\theta\{\theta\in A_\alpha(X)\}\ \ge\ 1-\alpha,\qquad \forall \theta\in\Theta.
\end{equation}
That is, $A_\alpha(X)$ is a $100(1-\alpha)\%$ confidence region for $\theta$.
Equivalently, for every assertion $H$,
the test that rejects $H$ when $\Pi_X(H)\le \alpha$ has Type~I error probability bounded by $\alpha$,
uniformly over $\theta\in H$, by \eqref{eq:validity-set}.

Moreover, combining validity with the credal-set characterization \eqref{eq:credal-alpha} gives a
useful interpretation of the probabilities $Q_x\in\mathcal{C}(\Pi_x)$: since validity makes each
$A_\alpha(X)$ a confidence region, any $Q_x\in\mathcal{C}(\Pi_x)$ assigns at least $1-\alpha$ probability
to each calibrated confidence region. As such the particular elements of $\mathcal{C}(\Pi_x)$ that are \emph{saturating} the $\alpha$-cuts, i.e. assign mass $Q_x(A_\alpha)=1-\alpha$ \emph{exactly}, are probability matching for genuine confidence sets, and can therefore be
viewed as (possibly conservative) \emph{confidence distributions} in a multivariate sense. For further discussions of the connection between confidence distributions and valid possibilistic inferential models, see e.g. \cite{reIM}. 

\subsubsection{Some constructions of valid IMs}\label{sec:validification}
The convenient representation of a possibility contour in \eqref{eq:contcutrep} via its $\alpha$-cuts means that \emph{any} nested family of confidence sets $\{C_\alpha(x)\}_{\alpha \in (0,1)}$ \emph{induces} a valid possibility contour by
\begin{equation*}
    \pi_x(\theta)=\sup \{\alpha\in(0,1):\theta \in C_\alpha(x)\}.
\end{equation*}
This means that \emph{most} of classical frequentist methods can be absorbed under the valid possibilistic umbrella (see Proposition \ref{prop:unbiased-equiv}; tests with nested acceptance regions can be similarly converted, see \cite[Appendix~G]{Martin2025}).

While the developments in this paper apply to general data-dependent possibility measures, such as those constructed via predictive random sets in \citep{Martin2013} and a chosen data-generating equation,
it is useful to recall a canonical IM construction that produces valid contours from familiar
likelihood-based rankings.

Given data $X=x$, let $L_x(\theta)$ denote the likelihood and define the relative likelihood
\[
  R(x,\theta)=\frac{L_x(\theta)}{\sup_{\vartheta\in\Theta} L_x(\vartheta)}\in[0,1].
\]
The function $\theta\mapsto R(x,\theta)$ is itself a (normalized) possibility contour and has a long history
as a likelihood-driven plausibility ranking (e.g. \cite{Dubois2024likelihood}), but it lacks is a \emph{standard scale} across problems.

A uniform scale can be obtained via \emph{validification}, i.e., a probability-to-possibility transform (see \citet{Martin2023b} for a full motivation of this construction, and \citet{Martin2025} for more of its developments).
One common choice is
\begin{equation}\label{eq:validified-contour}
  \pi_x(\theta)=P_\theta\{R(X,\theta)\le R(x,\theta)\},\qquad \theta\in\Theta,
\end{equation}
and then $\Pi_x(H)=\sup_{\theta\in H}\pi_x(\theta)$.
The contour in
\eqref{eq:validified-contour} then satisfies \eqref{eq:validity-point}, hence the resulting $\Pi_x$ is valid.
In many regular models, $\pi_x(\theta)$ coincides with the $p$-value function from a likelihood-ratio test of
$H_0:\Theta=\theta$.

The above construction admits refinements (e.g., conditioning on ancillary statistics or profiling
to eliminate nuisance parameters) that preserve validity while improving efficiency. 

\subsection{Why possibility?}\label{sec:why-imprecision}

The imprecise-uncertainty formalism, and possibility theory in particular, is not an aesthetic choice but a structural one. If we insist on frequentist calibration for arbitrary inferential output (validity), then \emph{imprecision is forced upon us}. Results such as the \emph{False confidence theorem} \citet{BalchMartinFerson2019FCT} show—that under mild regularity—any \emph{additive}, data-dependent probability measure, is unavoidably vulnerable to the phenomena of false confidence (roughly: systematically assigning high posterior probability to false assertions). This means that inferential procedures with ``familiar'' additive probabilistic output necessarily have to limit their scope in terms of assertions for which they can give reliability guarantees. Hence if we demand calibration guarantees for \emph{arbitrary} assertions, we \emph{must} therefore move beyond additivity to \emph{non-additive} (imprecise)
uncertainty quantification—i.e., coherent lower/upper probabilities—where strong calibration
properties are attainable in finite samples. 

A key identification made in \citep{Martin2023b}, is that for whichever imprecise-uncertainty formalism we decide to use, it should be expressive enough to support the natural calibration property of having credible sets that are confidence sets; moreover these credible sets should be easily identifiable from the uncertainty output. These natural requirements exclude, not by necessity, but rather by convenience, formalisms unable to satisfy the stronger form of reliability in \eqref{eq:validity-point}, leaving essentially only possibility as a viable option. The only exceptions are non-consonant upper probabilities induced by \emph{consistent random sets} \citep{DUBOIS1990consapprox}; for more, see the discussion in \citet{Martin2023b} Sections 3.1-3.2 leading up to Lemma 1. Importantly, for any such exception satisfying this strong form of validity, there exists a strongly valid possibility measure that is \emph{no less informative}. So in this statistical context, if we accept these structural and calibration requirements, we can not only safely restrict our attention to \emph{possibility measures}, but are guided to do so under efficiency considerations. Moreover, in the context of concern in this paper, with vacuous prior knowledge, these different notions of validity are equivalent, hence possibility is \emph{the canonical representation} of any valid formalism, as implied by Proposition \ref{prop:unbiased-equiv}, and more generally in \citet[Appendix~G]{Martin2025}. That is, when the inferential output of a procedure is a family of nested regions (e.g., acceptance regions, or confidence sets), there exists always a possibility measure representation of those regions that is no less informative/efficient. 

The strong validity requirement in \eqref{eq:validity-point} is not purely a matter of convenience. It implies a stronger \emph{uniform-in-hypotheses} version of \eqref{eq:validity-set}: 
\begin{equation}
    \sup_\theta P_\theta\{\Pi_X(H)\leq \alpha \; \text{for some $H\in \mathcal B(\Theta)$ with $\theta\in H$}\}\leq \alpha.
\end{equation}
For an inferential framework to admit an uncertainty calculus, some variant of this uniform calibration property is essential. If the form of uncertainty representation does not permit the user to assign degrees-of-belief post data, the framework would be dead on arrival as far as reliability concerns goes.

\subsection{Efficiency} \label{sec:efficiency-background}

The motivating question for this section is about \emph{efficiency}: among valid data-dependent possibility measures, which are the most informative?

A basic notion of informativeness in imprecise-probabilistic inference is induced by pointwise dominance.
Given two possibility measures $\Pi^1_x$ and $\Pi^2_x$ (with contours $\pi^1_x$ and $\pi^2_x$), if
\[
\Pi^1_x(H)\le \Pi^2_x(H)\quad\text{for all assertions $H$}
\qquad\text{equivalently}\qquad
\pi^1_x(\vartheta)\le \pi^2_x(\vartheta)\quad\text{for all $\vartheta$,}
\]
then $\Pi^1_x$ is uniformly sharper: its $\alpha$-cuts are nested inside those of $\Pi^2_x$ at every level $\alpha$, hence it yields uniformly smaller calibrated confidence regions.
When such dominance holds it provides a clear-cut efficiency comparison.
In general, however, dominance is too strong to expect: two valid contours can cross, being tighter in some regions or at some plausibility levels and looser in others.
This is one reason why a general, finite-sample efficiency theory for valid possibility measures has been difficult to formulate.

\subsubsection*{Asymptotic efficiency via a possibilistic Bernstein--von Mises theorem.}
Recent progress has clarified the large-sample picture.
\cite{Martin2025BvM} establish a possibilistic Bernstein--von Mises theorem for likelihood-based IMs.
Under standard Le Cam-type LAN regularity conditions (and some extra, mild uniform continuity assumptions), the IM contour $\pi_{X_n}$ from \eqref{eq:validified-contour}, after centering at the MLE $\hat\theta_{X_n}$ and scaling at the usual $n^{-1/2}$ rate, is uniformly approximated (locally, in $\theta$) by a Gaussian possibility contour with covariance matrix $(n I_\Theta)^{-1}$, i.e., the Cramér--Rao lower bound.
Equivalently, one can read their result as saying that, for large $n$,
\[
\pi_{X_n}(\theta)\ \approx\ \gamma_{\hat\theta_{X_n},\,J^{-1}_{X_n}}(\theta),
\]
where $\gamma_{m,\Sigma}$ is the probability-to-possibility transform of a Gaussian ${\rm N}(m,\Sigma)$ and $J_{X_n}$ is the observed information.
This implies that the IM's $\alpha$-cuts merge with the familiar likelihood-based elliptical confidence regions, and that there is \emph{no asymptotic efficiency loss} attributable to finite-sample validity or to the IM's inherent imprecision.
A notable interpretation given in \cite{Martin2025BvM} is credal-set based: asymptotically, the IM credal set is the smallest one that contains the efficient Gaussian law, so the IM is ``as tight as possible'' among consonant imprecise-probabilistic summaries compatible with that Gaussian limit.

The same work also treats nuisance-parameter elimination and resolves an efficiency question that had previously been addressed mainly empirically.
Specifically, \cite{Martin2025BvM} develop marginal possibilistic BvM results showing that profiling-based marginalization yields asymptotically tighter marginal contours than extension-based marginalization. In orthogonal settings, the profile-based marginal IM can even attain an adaptive efficiency property, matching the local information available in the known-nuisance ``gold standard'' case.

While the asymptotic efficiency result of \cite{Martin2025BvM} is specific to the calibrated LR IM construction in \eqref{eq:validified-contour}, the authors argue that this efficiency is not tied to the relative-likelihood construction per se. Rather,
they suggest that asymptotic efficiency should extend more broadly to IMs built from structured
associations and predictive random sets, where additional model structure is available. Formally
proving such a generalization is an open problem.

\subsubsection*{Finite-sample efficiency: how it has been assessed to date.}
In contrast to the now much clearer asymptotic story, finite-sample efficiency comparisons in the IM literature have mostly been handled on a case-by-case basis.
In applications and methodological papers, efficiency has typically been assessed by looking directly at the induced calibrated regions at one or a few confidence levels---e.g., comparing the length/volume of $A_\alpha(x)$ for a fixed $\alpha$---or by visually comparing contour plots for representative datasets.
Simulation studies often report coverage together with an average size metric (mean length/volume) for nominal regions, echoing the classical practice of comparing confidence procedures via expected size at fixed coverage.
For example, in the Behrens--Fisher illustration in \cite{reIM}, competing nominal 90\% intervals are compared by empirical coverage and mean length, with the IM-based construction achieving near-nominal coverage while improving the mean length relative to a more conservative competitor.

These comparisons are informative, but are either too subjective, or inherently \emph{level-specific} and do not by themselves define a general procedure-level ordering: conclusions can e.g. depend on which $\alpha$ is chosen, and where in the parameter space the procedure is evaluated.
Moreover, for IM constructions based on associations and predictive random sets, the main design choice controlling informativeness is the predictive random set itself. As emphasized in \cite{Martin2025BvM}, broad guidelines for selecting it to optimize statistical efficiency (alongside computational considerations) are still limited.

\section{Theory} \label{sec:theory}
\subsection{Choquet Loss and Choquet Risk}\label{sec:choclossandrisk}

Assume we have several data-dependent possibility measures available for summarizing inference
about a parameter $\theta \in \Theta$. Formally, an inferential procedure considered here is an
inferential model (IM) $x \mapsto \Pi_x$, where $\Pi_x$ is a possibility measure on
$(\Theta,\mathcal{B}(\Theta))$ with contour $\pi_x$.

To compare such procedures, we fix a non-negative penalty function
\[
(x,\theta,\vartheta) \longmapsto \Gamma(x,\theta,\vartheta) \in [0,\infty],
\]
interpreted as the cost of treating $\vartheta$ as a candidate value when the truth is $\theta$.
Two important special cases are:
(i) \emph{distance-to-truth} (accuracy) penalties, where $\Gamma(x,\theta,\vartheta)=\Gamma_\theta(\vartheta)$
and typically $\Gamma_\theta(\theta)=0$ (e.g., $\Gamma_\theta(\vartheta)=\|\vartheta-\theta\|^2$); and
(ii) \emph{concentration} penalties, where $\Gamma(x,\theta,\vartheta)=\Gamma_x(\vartheta)$, designed
to quantify the tightness of the induced plausibility regions, regardless of $\theta$.
In Section  \ref{sec:equi} we impose additional structural conditions on $\Gamma$
(e.g., invariance), and in \ref{sec:reparinvariance} we instead consider  coordinate-free penalties.

For a realized dataset $x$, the \emph{Choquet loss} of $\Pi_x$ at the true state $\theta$ is
the \emph{upper expectation} of $\Gamma(x,\theta,\cdot)$ with respect to $\Pi_x$, i.e. the Choquet integral:
\begin{equation*}
L(\theta,\Pi_x)
= \int \Gamma(x,\theta,\vartheta)\, d\Pi_x(\vartheta)
= \int_0^1 \sup_{\vartheta\in A_\alpha(x)} \Gamma(x,\theta,\vartheta)\, d\alpha ,
\end{equation*}
where the last equality follows by \eqref{eq:choquet-poss}.
Formally, we give the following definition. 

\begin{definition}[Choquet loss and Choquet risk]
Let \(x\mapsto \Pi_x\) be a data-dependent possibility measure on
\((\Theta,\mathcal B(\Theta))\), with contour \(\pi_x\) and \(\alpha\)-cuts
$
A_\alpha(x)=\{\vartheta\in\Theta:\pi_x(\vartheta)\ge \alpha\}$, $ \alpha\in[0,1]$.
Fix a non-negative measurable penalty $
(x,\theta,\vartheta)\mapsto \Gamma(x,\theta,\vartheta)\in[0,\infty]$.
For observed data \(x\) and true value \(\theta\), define the
\emph{Choquet loss} of \(\Pi_x\) by
\begin{equation}\label{eq:choquet-loss}
L(\theta,\Pi_x)
:= \int \Gamma(x,\theta,\vartheta)\, d\Pi_x(\vartheta)
= \int_0^1 \sup_{\vartheta\in A_\alpha(x)} \Gamma(x,\theta,\vartheta)\, d\alpha ,
\end{equation}
whenever the integral is well defined. The \emph{Choquet risk} of the
procedure \(\Pi=\{\Pi_x:x\in\mathcal X\}\) is the sampling expectation of the Choquet loss:
\begin{equation}\label{eq:choquet-risk}
R(\theta,\Pi) := \mathbb{E}_\theta\{L(\theta,\Pi_X)\},
\end{equation}
whenever the expectation exists.
\end{definition}

\begin{remark}\label{rem:data-truth-loss}
At first impression, allowing the loss-penalty $\Gamma$ to depend on both the data and the truth might seem
unconventional. Here, however, $\Gamma$ is not used to define a downstream decision rule (i.e.,
choosing an action under $\Pi_x$, as is considered in e.g. \citet{Martin2026decision}); it is used only to evaluate and compare inferential procedures.
For that purpose, data- and truth-dependence is natural.
\end{remark}


Because $\Pi_x$ is maxitive, the representation \eqref{eq:choquet-loss} shows that the
Choquet loss aggregates---across plausibility levels $\alpha$---the \emph{worst} penalty among parameter
values that remain at least $\alpha$-plausible under $\Pi_x$. Equivalently, since the Choquet integral is the upper expectation over the credal set \(C(\Pi_x)\), \(L(\theta,\Pi_x)\) is the least-favourable expected penalty among the precise probabilities compatible with the possibilistic summary. The sets $\{A_\alpha(X)\}_{\alpha\in[0,1]}$ are calibrated confidence regions, so the otherwise quite abstract Choquet integral (for general belief functions) reduces in \eqref{eq:choquet-loss} to a simple average of ``worst-offender''
penalties over a nested family of confidence sets.

\paragraph{A risk reduction via $\alpha$-cuts.} When $L(\theta,\Pi_X)$ is integrable, Tonelli/Fubini allows us to interchange the outer expectation and the
$\alpha$-integral, a key reduction that will be exploited repeatedly.
Let $\Pi_X$ have contour $\pi_X$ and $\alpha$-cuts $A_\alpha(X)$.
Define the $\alpha$-cut risk by
\begin{equation}\label{eq:alpha-risk}
r_\alpha(\theta,\Pi)
:= \mathbb{E}_\theta\!\left[\sup_{\vartheta\in A_\alpha(X)} \Gamma(X,\theta,\vartheta)\right].
\end{equation}
Since the integrand in \eqref{eq:choquet-loss} is non-negative, Tonelli's theorem gives the decomposition
\begin{equation}\label{eq:risk-decomposition}
R(\theta,\Pi)=\int_0^1 r_\alpha(\theta,\Pi)\, d\alpha.
\end{equation}

\paragraph{Concentration penalties.} \leavevmode\par
A common notion of finite-sample efficiency compares procedures through the tightness of their
confidence/plausibility regions. The following penalties encode that idea directly.
Assume $\Theta\subseteq\mathbb{R}^d$ and that the $\alpha$-cuts $A_\alpha(x)$ have finite Lebesgue measure
$|A_\alpha(x)|$. Let $\phi:[0,\infty)\to[0,\infty)$ be non-decreasing (penalizing larger sets). Define the data-dependent penalty
\begin{equation}\label{eq:concentration-penalty}
\Gamma_x(\vartheta) := 
\phi\!\Bigl(\bigl|A_{\pi_x(\vartheta)}(x)\bigr|\Bigr).
\end{equation}
This penalty depends on \(\vartheta\) only through its contour level
\(t=\pi_x(\vartheta)\) and the size of the corresponding \(t\)-cut.
Because the cuts are nested, \(A_t(x)\subseteq A_s(x)\) whenever
\(t\ge s\), the map \(t\mapsto |A_t(x)|\) is non-increasing. Since
\(\phi\) is non-decreasing, it follows that
\[
t\mapsto \phi(|A_t(x)|)
\]
is non-increasing. Therefore, for any \(s\in[0,1]\),
\[
\sup_{\vartheta:\pi_x(\vartheta)\ge s}
\Gamma_x(\vartheta)
=
\sup_{t\ge s}\phi(|A_t(x)|)
=
\phi(|A_s(x)|).
\]
Consequently the Choquet loss collapses to an ordinary integral of cut sizes
\begin{equation}\label{eq:choquet-loss-collapse}
L(\theta,\Pi_x) = \int_0^1 \phi(|A_s(x)|)\, ds.
\end{equation}
Moreover, for $\phi(\ell)=\ell$, the concentration loss reduces to 
\begin{equation}
    \int_0^1|A_\alpha(x)|d\alpha=\int_\Theta\int_0^1\mathbf1\{\pi_x(\vartheta)\geq \alpha\}d\alpha d\vartheta=\int_\Theta\pi_x(\vartheta)d\vartheta.
\end{equation}
This makes it geometrically explicit that the concentration loss is simply the volume under the possibility contour.

\subsection{Optimal Choquet Risk}\label{sec:optRisk}

Let $\mathscr C_{\Pi}$  denote a collection of valid data-dependent possibility procedures, i.e. a set of distinct
valid procedures $\{\Pi^1,\Pi^2,\ldots\}$ where each procedure is a family
$\Pi^k=\{\Pi_x^k:x\in\mathbb{X}\}$. Fix a penalty $\Gamma$; all risk, optimality, and admissibility statements
in this section are understood with respect to this fixed $\Gamma$.

For $\Pi\in\mathscr C_{\Pi}$, recall the Choquet risk functional $R(\theta,\Pi)$ in \eqref{eq:choquet-risk}.
We say that $\Pi^\ast\in\mathscr C_{\Pi}$ is \emph{Choquet-risk minimal at $\theta$} if
\[
R(\theta,\Pi^\ast) \le R(\theta,\Pi)\qquad \text{for all }\Pi\in\mathscr C_{\Pi}.
\]
If $\Pi^\ast$ is risk minimal at every $\theta\in\Theta$, we call it \emph{Choquet-risk optimal}
(pointwise optimal/minimal).
We say $\Pi\in\mathscr C_{\Pi}$ is \emph{inadmissible} if there exists $\Pi'\in\mathscr C_{\Pi}$ such that
$R(\theta,\Pi')\le R(\theta,\Pi)$ for all $\theta$ with strict inequality for at least one $\theta$;
otherwise $\Pi$ is admissible. 

A very simple but powerful tool is a direct consequence of the $\alpha$-cut risk reduction.
\begin{lemma}\label{lem:alpha-cut-minimal}
Fix $\theta\in\Theta$ and a class $\mathscr C_{\Pi}$. Suppose $\Pi^\ast\in\mathscr C_{\Pi}$ satisfies
\[
r_\alpha(\theta,\Pi^\ast)\le r_\alpha(\theta,\Pi)\qquad
\text{for all }\Pi\in\mathscr C_{\Pi}\text{ and all }\alpha\in[0,1]
\]
(or, more generally, for Lebesgue-a.e.\ $\alpha$). Then $\Pi^\ast$ is Choquet-risk minimal at $\theta$
over $\mathscr C_{\Pi}$. If the inequality is strict on a set of $\alpha$'s of positive Lebesgue measure for every
$\Pi\neq \Pi^\ast$, then $\Pi^\ast$ is the unique risk minimizer at $\theta$.
\end{lemma}

\begin{proof}
By \eqref{eq:risk-decomposition},
$R(\theta,\Pi)=\int_0^1 r_\alpha(\theta,\Pi)\, d\alpha$.
Integrating the pointwise inequality in $\alpha$ yields
$R(\theta,\Pi^\ast)\le R(\theta,\Pi)$ for all $\Pi\in\mathscr C_{\Pi}$.
Strictness on a set of positive measure implies strict inequality after integration.
\end{proof}

\begin{remark}\label{rem:lemma-use}
Lemma \ref{lem:alpha-cut-minimal} is most useful when $\Pi_X$ is constructed by stacking a nested
family of confidence sets with established optimality properties at each level $\alpha$. In that case one can
often verify $\alpha$-cut optimality directly (at the level-set scale), then lift it to a Choquet-risk comparison
by integrating over $\alpha$. When the class $\mathscr C_{\Pi}$ is defined through distribution-level properties
(e.g.\ validity, unbiasedness, or equivariance), some bookkeeping may be needed to translate those constraints into the corresponding
constraints on $\alpha$-cuts before applying the lemma.
\end{remark}

\paragraph{Concentration-optimality via confidence-set size.}
Consider the concentration penalty $\Gamma_x$ in \eqref{eq:concentration-penalty}.
Then $r_\alpha(\theta,\Pi)=\mathbb{E}_\theta[\phi(|A_\alpha(X)|)]$, so optimality for the Choquet
risk reduces to optimality of expected (transformed) set size at each confidence level.

\begin{corollary}\label{cor:concentration-optimal}
Fix $\phi$ as in \eqref{eq:concentration-penalty}, and consider a class $\mathscr C_{\Pi}$ induced by a class $\mathcal{C}$
of nested families of confidence sets $\{A_\alpha(\cdot)\}_{\alpha\in[0,1]}$.
Suppose there exists $\Pi^\ast\in\mathscr C_{\Pi}$ with $\alpha$-cuts $A_\alpha^\ast(X)$ such that, for all $\Pi\in\mathscr C_{\Pi}$ with $\alpha$-cuts $A_\alpha(X)$,
\[
\mathbb{E}_\theta\!\left[\phi\bigl(|A_\alpha^\ast(X)|\bigr)\right]
\le
\mathbb{E}_\theta\!\left[\phi\bigl(|A_\alpha(X)|\bigr)\right],
\qquad \text{for all $\alpha\in[0,1]$, and $\theta\in\Theta$}.
\]
Then $\Pi^\ast$ is Choquet-risk optimal within $\mathscr C_{\Pi}$ for the concentration penalty \eqref{eq:concentration-penalty}.
\end{corollary}

\begin{proof}
Under $\Gamma_x$, the $\alpha$-risk is
$r_\alpha(\theta,\Pi)=\mathbb{E}_\theta[\phi(|A_\alpha(X)|)]$.
The assumed inequality implies $r_\alpha(\theta,\Pi^\ast)\le r_\alpha(\theta,\Pi)$ for all $\alpha$,
and Lemma \ref{lem:alpha-cut-minimal} yields the conclusion.
\end{proof}

The concentration penalty \eqref{eq:concentration-penalty}
quantifies efficiency in terms of the tightness of the plausibility regions $\{A_\alpha(X)\}$. The special case $\phi(\ell)=\ell$ recovers standard integrated expected volume (in one dimension,
integrated expected length). Corollary \ref{cor:concentration-optimal} thus provides a direct bridge from
classical confidence-set size optimality (at each $\alpha$) to Choquet-risk optimality of the corresponding
possibility procedures.

\begin{example}[Normal mean with known variance]\label{ex:normal-known-variance}
Let $X_1,\ldots,X_n \stackrel{\text{iid}}{\sim} \mathcal{N}(\mu,\sigma^2)$ with $\sigma^2$ known and $\Theta=\mathbb{R}$.
For each $\alpha\in(0,1)$, the usual $100(1-\alpha)\%$ confidence interval
\[
A_\alpha(X)
=
\left[\bar X - z_{1-\alpha/2}\frac{\sigma}{\sqrt{n}},\;
      \bar X + z_{1-\alpha/2}\frac{\sigma}{\sqrt{n}}\right]
\]
is nested in $\alpha$ and has constant length $|A_\alpha(X)| = 2 z_{1-\alpha/2}\sigma/\sqrt{n}$.
Stacking $\{A_\alpha(X)\}$ yields a valid contour $\pi_X$ whose $\alpha$-cuts recover these intervals; this is the typical two-sided z-test p-value function.
Within the class of translation-equivariant (hence constant-length) $(1-\alpha)$ confidence intervals,
this choice minimizes $|A_\alpha(X)|$ for each fixed $\alpha$; therefore, for $\phi(\ell)=\ell$,
Corollary \ref{cor:concentration-optimal} implies that the induced possibility procedure is Choquet-risk optimal
for the integrated expected length criterion. 

This is a very simple example. Indeed, not only are intervals $A_\alpha(X)$ minimal in \emph{expected} size here, but also minimal for each observed x.
\end{example}

For contour concentration there is a clear translation of classical optimality theory into possibilistic language. For distance-to-truth penalties this is less direct however. Whilst arguably more descriptive, these penalties seemingly demand more structure — or optimality in terms of pointwise risk minimality might need to be relaxed into a minimax principle: the content of Section  \ref{sec:equi}. When analytical comparisons are unavailable, direct comparisons can also be made. This is discussed in Section  \ref{sec:directComp}.

\subsection{Confidence risk and Choquet risk}\label{sec:confrisk}

Schweder and Hjort's \citep{schweder_hjort_2016} notion of \emph{confidence risk} provides a natural \emph{additive} point of comparison for the Choquet-risk criteria developed in the preceding sections.
In the scalar case, let $Q_x$ be a confidence distribution (CD) on $\Theta\subseteq\mathbb R$.
For a non-negative penalty $\gamma_\theta(\vartheta)\ge 0$, define the
\emph{confidence loss} 
\begin{equation}\label{eq:cd-loss}
  L_{\mathrm{CD}}(\theta,Q_x)
  \;=\;
  \int_\Theta \gamma_\theta(\vartheta)\,Q_x(d\vartheta),
\end{equation}
and the corresponding \emph{confidence risk} 
\begin{equation}\label{eq:cd-risk}
  R_{\mathrm{CD}}(\theta,Q)
  \;=\;
  \mathbb E_\theta\!\left\{L_{\mathrm{CD}}(\theta,Q_X)\right\}
  \;=\;
  \mathbb E_\theta\!\left[\int_\Theta \gamma_\theta(\vartheta)\,Q_X(d\vartheta)\right].
\end{equation}
This is the \emph{precise-probability} analogue of Choquet loss and risk: $Q_x$ yields an ordinary expectation,
whereas $\Pi_x$ yields an upper expectation over its $\alpha$-cuts.


    For the \emph{concentration-penalty} $\Gamma_x$ in \eqref{eq:concentration-penalty}, the connection is exact. Let
\[
\mathcal C^*(\Pi_x)
  := \{Q_x\in\mathcal C(\Pi_x): Q_x(A_\alpha(x)) = 1-\alpha,\ \forall \alpha\in[0,1]\}
\]
denote the class of \emph{saturating inner probabilities}. In \citep{reIM} it is shown that any $Q_x\in \mathcal C^*(\Pi_x)$ has a uniform mixture representation  
    \begin{equation}
        Q_x(\cdot)=\int_0^1K_x^t(\cdot)dt,
    \end{equation}
with kernels $K_x^t$ supported on the level set $\{\pi_x=t\}$. By the mixture distribution law of total expectation, we get
    \begin{equation*}
        E_{Q_x}[\Gamma_x]=\int_0^1E_{K^t_x}[\Gamma_x(\vartheta]dt.
    \end{equation*}
    But $\pi_x(\vartheta)=t$ on the support of $K_x^t$, hence $\Gamma_x(\vartheta)$ is \emph{constant} $K_x^t$-a.s. Therefore
    \begin{equation*}
        E_{K_x^t}[\Gamma_x]=\phi(|A_t(x)|).
    \end{equation*}
    So 
    \begin{equation*}
        E_{Q_x}[\Gamma_x]=\int_0^1\phi(|A_t(x)|)dt = L(\theta, \Pi_x).
    \end{equation*} 
Hence, for concentration penalties,
\begin{equation*}
    R(\theta, \Pi)=E_\theta E_{Q_X}[\Gamma_X], \qquad Q_X\in \mathcal C^*(\Pi_X).
\end{equation*}
Thus Choquet risk coincides with confidence risk after replacing the additive confidence distribution by
any saturating inner of the possibility measure. In particular, if
\[
\pi_x(\vartheta)=1-|2Q_x(\vartheta)-1|
\]
is the contour induced by a scalar confidence curve, then Choquet risk evaluates the expected size of
the associated \emph{central} confidence intervals, not only one-sided bounds, which confidence risk is typically restricted to.

For general distance-to-truth penalties, this exact identification fails. Instead
\begin{equation*}
    L(\theta, \Pi_x)=\sup_{Q_x\in\mathcal C(\Pi_x)}E_{Q_x}[\Gamma_\theta(\vartheta)],
\end{equation*}
so Choquet loss is the upper envelope of the corresponding ``confidence losses'' over all inner
probabilities dominated by $\Pi_x$. In this sense, Choquet loss is a robustified, or least-favourable,
confidence loss. Under some mild regularity, this picture is made clearer.

\begin{proposition}
\label{prop:lf-saturating-radial}
Let $\pi_x:\mathbb R^d\to[0,1]$ be continuous and unimodal, attaining its unique maximum at
$\hat\theta_x$, and suppose $\pi_x(\vartheta)\to0$ as $\|\vartheta\|\to\infty$, so that the cuts
$A_\alpha(x)=\{\vartheta:\pi_x(\vartheta)\ge \alpha\}$ are compact. Let
\[
\Gamma_\theta(\vartheta)=g(\|\vartheta-\theta\|),
\]
where $g:[0,\infty)\to[0,\infty)$ is continuous and strictly increasing, and assume
$L(\theta,\Pi_x)<\infty$. Then:

\begin{enumerate}
\item[(a)] The maximizer set
\[
M_{\alpha,\theta}(x):=\arg\max_{\vartheta\in A_\alpha(x)}\Gamma_\theta(\vartheta)
\]
is nonempty and satisfies $M_{\alpha,\theta}(x)\subseteq \{\vartheta:\pi_x(\vartheta)=\alpha\}$ for every $\alpha\in(0,1)$.

\item[(b)] There exists a (generally $\theta$-dependent) \emph{saturating} inner $Q^*_{x,\theta}\in\mathcal{C}^*(\Pi_x)$ such that
\[
E_{Q^*_{x,\theta}}\{\Gamma_\theta(\vartheta)\}=L(\theta, \Pi_x).
\]

\item[(c)] Consequently,
\[
\sup_{Q\in\mathcal{C}(\Pi_x)}E_Q\{\Gamma_\theta(\vartheta)\}
\;=\;
\sup_{Q\in\mathcal{C}^*(\Pi_x)}E_Q\{\Gamma_\theta(\vartheta)\}
\;=\;L(\theta, \Pi_x).
\]
\end{enumerate}
\end{proposition}

\begin{proof}
Fix $\alpha\in(0,1)$. Since $A_\alpha(x)$ is compact and $\Gamma_\theta$ is continuous, the maximum is attained and
$M_{\alpha,\theta}(x)\neq\emptyset$.

To show maximizers lie on the $\alpha$-level set, note that strict monotonicity of $g$ implies that maximizing
$\Gamma_\theta(\vartheta)$ over $A_\alpha(x)$ is equivalent to maximizing $\|\vartheta-\theta\|$ over $A_\alpha(x)$.
Let $\vartheta^*\in M_{\alpha,\theta}(x)$. If $\vartheta^*$ were an interior point of $A_\alpha(x)$, then there exists
$\varepsilon>0$ such that the Euclidean ball $B(\vartheta^*,\varepsilon)\subseteq A_\alpha(x)$.
For $\delta>0$ small enough, $\vartheta':=\vartheta^*+\delta(\vartheta^*-\theta)$ belongs to $B(\vartheta^*,\varepsilon)$,
hence to $A_\alpha(x)$, while $\|\vartheta'-\theta\|>\|\vartheta^*-\theta\|$, contradicting maximality.
Therefore $M_{\alpha,\theta}(x)\subseteq \partial A_\alpha(x)$.

Finally, since $A_\alpha(x)=\{\pi_x\ge \alpha\}$ and $\pi_x$ is continuous, any boundary point of $A_\alpha(x)$ must satisfy
$\pi_x(\vartheta)=\alpha$. This proves (a).

For (b), choose a Borel-measurable selector $\alpha\mapsto \vartheta^*_{\alpha,\theta}(x)\in M_{\alpha,\theta}(x)$
for $\alpha\in(0,1)$ (existence follows from a standard measurable maximum/selection theorem since
$A_\alpha(x)$ is compact and the criterion is continuous).
Let $U\sim \mathrm{Unif}(0,1)$ and define $\Theta^*=\vartheta^*_{U,\theta}(x)$. Write $Q^*_{x,\theta}$ for the law of $\Theta^*$.
By (a), $\pi_x(\Theta^*)=\pi_x(\vartheta^*_{U,\theta}(x))=U$ almost surely, and therefore for any $\alpha\in[0,1]$,
\[
Q^*_{x,\theta}(A_\alpha(x)) = P(\pi_x(\Theta^*)\ge \alpha)=P(U\ge \alpha)=1-\alpha,
\]
so $Q^*_{x,\theta}$ is saturating.

Moreover, for any Borel set $B$,
\[
\{\Theta^*\in B\}\subseteq \{U=\pi_x(\Theta^*)\le \sup_{\vartheta\in B}\pi_x(\vartheta)\}=\{U\le \Pi_x(B)\},
\]
and hence
\[
Q^*_{x,\theta}(B)\le P(U\le \Pi_x(B))=\Pi_x(B),
\]
so $Q^*_{x,\theta}\in\mathcal{C}(\Pi_x)$, i.e. $Q^*_{x,\theta}\in\mathcal{C}^*(\Pi_x)$.

Finally,
\[
E_{Q^*_{x,\theta}}\{\Gamma_\theta(\vartheta)\}
=E\bigl[\Gamma_\theta(\vartheta^*_{U,\theta}(x))\bigr]
=\int_0^1 \Gamma_\theta(\vartheta^*_{\alpha,\theta}(x))\,d\alpha
=\int_0^1 \sup_{\vartheta\in A_\alpha(x)}\Gamma_\theta(\vartheta)\,d\alpha
=L(\theta, \Pi_x),
\]
which establishes (b).

For (c), recall that the Choquet integral of $\Gamma_\theta$ w.r.t.\ $\Pi_x$
coincides with the upper expectation over $\mathcal{C}(\Pi_x)$, i.e.
\[
\sup_{Q\in\mathcal{C}(\Pi_x)}E_Q\{\Gamma_\theta(\vartheta)\}=L(\theta, \Pi_x).
\]
Since $\mathcal{C}^*(\Pi_x)\subseteq\mathcal{C}(\Pi_x)$ we always have
$\sup_{\mathcal{C}^*}E_Q\Gamma_\theta\le \sup_{\mathcal{C}}E_Q\Gamma_\theta$; but (b) shows that $L(\theta, \Pi_x)$ is
attained by some $Q^*_{x,\theta}\in\mathcal{C}^*(\Pi_x)$. Hence both suprema equal $L(\theta, \Pi_x)$.
\end{proof}

\begin{remark}
The optimizer $Q^*_{x,\theta}$ is least-favourable and typically depends on $\theta$; the proposition does \emph{not}
assert that a single (fixed) saturating inner attains the upper expectation for all $\theta$ simultaneously. This typically would only happen for a concentration penalty.
\end{remark}

\begin{remark}
    The argument in Proposition \ref{prop:lf-saturating-radial} does not require Euclidean radiality per se. The key step is that if $\vartheta^*\in int(A_\alpha)$, then there exists a nearby $\vartheta'\in A_\alpha$ with $\Gamma_\theta(\vartheta')>\Gamma_\theta(\vartheta^*)$. When $\Gamma_\theta(\vartheta)=g(\|\vartheta-\theta\|)$ with g strictly increasing, we ensure this by moving a tiny step away from $\theta$ along the ray through $\vartheta^*$. So we can replace $\|\vartheta-\theta\|$ by a``distance-like''function $\rho(\vartheta,\theta)$ provided $\rho$ has the analogous ray-monotonicity property. This should be satisfied by any norm-induced distance (e.g. Mahalanobis).
\end{remark}

The scalar confidence-curve case makes the comparison especially transparent. Suppose
$\Theta\subseteq\mathbb R$ and $\pi_x$ is induced by an exact confidence distribution $Q_x$ via
\[
\pi_x(\vartheta)=1-|2Q_x(\vartheta)-1|,
\]
with quantile function $q_x$. Then
\[
A_\alpha(x)=\bigl[q_x(\alpha/2),\,q_x(1-\alpha/2)\bigr].
\]
Writing $L_\alpha(x)=q_x(\alpha/2)$ and $U_\alpha(x)=q_x(1-\alpha/2)$, we obtain
\[
L(\theta,\Pi_x)
  = \int_0^1 \max\{\Gamma_\theta(L_\alpha(x)),\Gamma_\theta(U_\alpha(x))\}\,d\alpha,
\]
whereas
\[
L_{\mathrm{CD}}(\theta,Q_x)
  = \frac12 \int_0^1
      \{\Gamma_\theta(L_\alpha(x))+\Gamma_\theta(U_\alpha(x))\}\,d\alpha.
\]
Therefore
\[
L(\theta,\Pi_x)\ge L_{\mathrm{CD}}(\theta,Q_x),
\]
with equality iff the two endpoints incur equal penalty for Lebesgue-a.e.\ $\alpha$, which essentially only happens for a non $\theta$-dependent penalty, such as for the concentration penalty case we saw earlier. Choquet loss thus
replaces the boundarywise average of confidence loss by the boundarywise maximum, which is exactly
the robustification induced by passing from additive to possibilistic uncertainty.

\section{Optimality under constraints}
\subsection{Unbiased possibility measures}\label{sec:pos-unb}

Pointwise Choquet-risk optimality is generally unavailable over the full class of valid possibility measures. In classical confidence theory, unbiasedness identifies a smaller class within which finite-sample optimality statements can be made. For this purpose we introduce a strong condition on the orientation of the inherent plausibility order towards the truth, such that probability of assigning small plausibility to any point $\vartheta$ is minimal when that point is the data generating parameter. This condition, once validity is in place, is shown to be exactly the possibilistic version of classical unbiasedness. 

\begin{definition}[Possibilistic unbiasedness]\label{def:poss-unbiased}
Let $\Pi_X$ be a data-dependent possibility measure on $(\Theta,\mathcal{A})$ with contour $\pi_X$. 
We say that $\Pi_X$ is \emph{unbiased} if, for every fixed $\theta\in\Theta$, the random variable $\pi_X(\theta)$ is stochastically largest under $P_\theta$. 
Equivalently,
\begin{equation}\label{eq:poss-unbiased}
P_\theta\{\pi_X(\theta)\le \alpha\}\ \le\ P_\vartheta\{\pi_X(\theta)\le \alpha\} \quad \text{for all $\vartheta\in\Theta$ and all $\alpha\in[0,1]$}.
\end{equation}
\end{definition}

 Given the ubiquitousness of unbiasedness, it should come as no surprise that other efforts to produce optimal inferential procedures have arrived at similar formulations. The novelty in definition \ref{def:poss-unbiased} is for the most part just a clean formulation in possibilistic language.

\begin{remark}[Connections to earlier IM optimality and unbiased confidence curves]
\label{rem:unbiased-related}
The stochastic ordering in Definition~\ref{def:poss-unbiased} is closely related to previous
attempts to formalize ``centering'' and ``efficiency'' under calibration constraints.

In the original IM optimality development for two-sided point assertions, \citet[(4.2)]{Martin2013}
impose the requirement that the belief assigned to the complement $\{\theta_0\}^c$ be
(stochastically) smallest under the model $\theta=\theta_0$, and they note a connection to the
classical unbiasedness condition for two-sided tests. In the consonant/possibilistic setting,
this is the same ordering up to a monotone reparameterization: writing
$\lPi_x$ for the necessity measure associated with $\Pi_x$, one has
\[
\lPi_x(\{\psi\}^c)=1-\Pi_x(\{\psi\})=1-\pi_x(\psi),
\]
so a stochastic dominance statement for $\lPi_X(\{\psi\}^c)$ is equivalent to one for $\pi_X(\psi)$. 

A closely related notion appears in the confidence-curve literature. \cite{Schweder2018Unbiased} defines an
\emph{unbiased confidence curve} $cc(\theta;X)$ by requiring $cc(\theta_1;X)$ to be stochastically
larger than $U(0,1)$ for false $\theta_1$, and shows how this implies unbiasedness of the induced
confidence sets and associated tests. With the identification $cc(\theta;x)=1-\pi_x(\theta)$,
Schweder's ordering matches Definition~\ref{def:poss-unbiased} (under exact validity at the truth). 

\end{remark}

\begin{example}[Normal mean unbiased ranking ]\label{ex:norm-mean-unb-ranking}
Consider the normal location model $X\sim N(\theta,1)$. 
The relative likelihood
\[
q_x(\psi)\ :=\ R(x,\psi)
=\frac{L_x(\psi)}{\sup_{\vartheta}L_x(\vartheta)}
=\exp\!\left\{-\tfrac12(x-\psi)^2\right\}
\]
is a natural likelihood-based \emph{ranking} of parameter values, and it is itself a possibility contour (normalized to $\sup_\psi q_x(\psi)=1$). 
Moreover, for each fixed $\psi$, the random variable $q_X(\psi)=\exp\{-\tfrac12(X-\psi)^2\}$ is stochastically \emph{largest} under $P_\psi$: under $P_\vartheta$ we have $X-\psi\sim N(\vartheta-\psi,1)$, so $|X-\psi|$ is stochastically smallest at $\vartheta=\psi$, and $q_X(\psi)$ is a decreasing function of $|X-\psi|$. 
Hence $q$ is \emph{unbiased} in the sense of Definition \ref{def:poss-unbiased} (with $q$ in place of $\pi$).

The raw order $q_X$ is correctly oriented, but it is not yet valid. A useful device in the IM literature is to \emph{validify} a preliminary ranking $q_X(\psi)$ via a probability-to-possibility transform. Validification replaces the raw unbiased ranking by a calibrated one:
\[
\pi_x(\psi)\ :=\ P_\psi\{q_X(\psi)\le q_x(\psi)\}
= P_\psi\{|X-\psi|\ge |x-\psi|\}
=2\{1-\Phi(|x-\psi|)\},
\]
i.e., the usual two-sided $z$-test p-value. 
This $\pi_x$ is exactly valid and, since validification is a monotone transformation applied pointwise in $\psi$, it preserves the stochastic ordering that defines unbiasedness. So the resulting contour — the calibrated likelihood ratio — is both valid and unbiased.

\end{example}

\begin{remark}
    This ``conceptual separation'' of validity from unbiasedness, allows us to consider transformations of the raw evidence order that ``disentangles'' any perversions, and returns an unbiased contour which then can be validified. ``Debiasing'' of raw contours $q_X$ will be investigated in a future project.
\end{remark}

Recall the familiar duality: unbiased level-$\alpha$ tests correspond (via inversion of acceptance regions) to unbiased $(1-\alpha)$ confidence sets; see, e.g., \citet{lehmann2005testing}. 
Unbiasedness can be expressed either as a power inequality for tests, or as a “false coverage probability” inequality for the corresponding confidence sets. 
The next proposition formalizes an equivalence between this classical unbiasedness and possibilistically unbiased contours, under the usual validity/coverage interpretation of $\alpha$-cuts.

\begin{proposition}\label{prop:unbiased-equiv}
Let $\{C_\alpha(X):\alpha\in[0,1]\}$ be a family of confidence sets satisfying:
\begin{enumerate}
\renewcommand{\labelenumi}{\roman{enumi}.}
\item \emph{Nestedness:} if $\alpha_1<\alpha_2$ then $C_{\alpha_2}(X)\subseteq C_{\alpha_1}(X)$ a.s.;
\item \emph{Nominal coverage:} for all $\theta\in\Theta$ and $\alpha\in[0,1]$,
\begin{equation}\label{eq:nominal-coverage}
P_\theta\{\theta\in C_\alpha(X)\}\ \ge\ 1-\alpha;
\end{equation}
\item \emph{Unbiasedness:} for all $\theta,\vartheta\in\Theta$ and $\alpha\in[0,1]$,
\begin{equation}\label{eq:unbiased-C}
P_\theta\{\theta\in C_\alpha(X)\}\ \ge\ P_\vartheta\{\theta\in C_\alpha(X)\}.
\end{equation}
\end{enumerate}
Define a contour by stacking
\[
\pi_X(\vartheta):=\sup\{\alpha\in[0,1]:\ \vartheta\in C_\alpha(X)\}.
\]
Assume the mild boundary regularity that, for each fixed $\vartheta$, the map $\alpha\mapsto 1\{\vartheta\in C_\alpha(X)\}$ is a.s.\ left-continuous (equivalently, the $C_\alpha(X)$ can be taken closed). Then:
\begin{enumerate}
\renewcommand{\labelenumi}{\alph{enumi})}
\item the $\alpha$-cuts recover the confidence sets: $\{\vartheta:\pi_X(\vartheta)\ge \alpha\}=C_\alpha(X)$ a.s.\ for all $\alpha$;
\item the contour is strongly valid: $P_\theta\{\pi_X(\theta)<\alpha\}\le \alpha$ for all $\theta,\alpha$;
\item the induced possibility measure is unbiased in the sense of Definition~\ref{def:poss-unbiased}.
\end{enumerate}
Conversely, if a contour $\pi_X$ is valid and possibilistically unbiased, then its $\alpha$-cuts
$$A_\alpha(X)=\{\vartheta:\pi_X(\vartheta)\ge \alpha\}$$
form a nested family of confidence sets satisfying \eqref{eq:nominal-coverage}--\eqref{eq:unbiased-C}.
\end{proposition}

\begin{proof}
For (a), fix $\alpha$ and $\vartheta$. By definition, $\pi_X(\vartheta)\ge \alpha$ iff $\vartheta\in C_\beta(X)$ for some $\beta\ge \alpha$. 
By nestedness, $\vartheta\in C_\beta(X)$ for some $\beta\ge \alpha$ is equivalent to $\vartheta\in C_\alpha(X)$; the left-continuity assumption rules out boundary pathologies at the level $\alpha$. Hence $\{\vartheta:\pi_X(\vartheta)\ge \alpha\}=C_\alpha(X)$ a.s.

For (b), by (a), $\{\pi_X(\theta)<\alpha\}=\{\theta\notin C_\alpha(X)\}$ a.s., so
\[
P_\theta\{\pi_X(\theta)<\alpha\}=1-P_\theta\{\theta\in C_\alpha(X)\}\le \alpha
\]
by \eqref{eq:nominal-coverage}.

For (c), fix $\psi\in\Theta$. Again by (a), $\{\pi_X(\psi)<\alpha\}=\{\psi\notin C_\alpha(X)\}$ a.s. Hence,
\[
P_\psi\{\pi_X(\psi)<\alpha\}
=1-P_\psi\{\psi\in C_\alpha(X)\}
\le 1-P_\vartheta\{\psi\in C_\alpha(X)\}
=P_\vartheta\{\pi_X(\psi)<\alpha\},
\]
using \eqref{eq:unbiased-C}. The inequality for “$\le$” in \eqref{eq:poss-unbiased} follows by taking limits $\alpha'\downarrow\alpha$.

For the converse, take $C_\alpha(X):=A_\alpha(X)$. Nestedness is immediate from the definition of $\alpha$-cuts. 
Nominal coverage follows from validity. Unbiasedness \eqref{eq:unbiased-C} is exactly Definition~\ref{def:poss-unbiased} rewritten in terms of membership events, since
$P_\theta\{\psi\in A_\alpha(X)\}=1-P_\theta\{\pi_X(\psi)<\alpha\}$.
\end{proof}

\begin{remark}
    A variant of part a) $\Leftrightarrow$ b) above is already established in \citet{Martin2025}[Appendices G]; it is included here for completeness, and for a minor change in convention. Here we are defining $\alpha$-cuts as super-level sets of the contour, and validity defined by strict inequality, to ensure the $\alpha$-cuts themselves are the nested coverage-correct confidence sets. This difference disappears for continuous data.
\end{remark}


\begin{remark}\label{rem:schweder}
\cite{Schweder2018Unbiased} (see, e.g., his ``confidence curve'' viewpoint) conjectures that likelihood / deviance-driven confidence curves are unbiased. 
If such a statement held in full generality, then by Proposition~\ref{prop:unbiased-equiv} it would imply that LR/profile-LR tests are always unbiased. 
Classical theory does not support such a general conclusion: exactness and unbiasedness typically require additional structure (often through conditioning in exponential families), and LR procedures can be biased in finite samples without it. 
This suggests that deviance-based unbiasedness is best viewed as a property of structured families (or an asymptotic phenomenon), rather than a universal theorem.
\end{remark}

\subsubsection{Optimal unbiased possibility measures}
The equivalence in Proposition \ref{prop:unbiased-equiv} is exactly what is needed to translate classical confidence optimality into a possibilistic framework.

The “optimal unbiased” case is an immediate specialization of the $\alpha$-cut risk decomposition in Lemma \ref{lem:alpha-cut-minimal}. 
Under standard monotone likelihood ratio (MLR) conditions, classical theory provides \emph{most accurate unbiased} confidence intervals, obtained by inverting UMPU tests; see \citet[Lemma~5.5.1]{lehmann2005testing}. 
Via the Ghosh--Pratt identity (e.g., \citealp{Pratt01091961}), most-accurate unbiasedness corresponds to minimal expected length among unbiased intervals, and therefore Corollary~\ref{cor:concentration-optimal} lifts this to a concentration-optimality statement for possibilistic procedures.

\begin{corollary}\label{cor:optimal-unbiased}
Assume $\{P_\theta:\theta\in\Theta\subset\mathbb{R}\}$ admits a real-valued statistic $T(X)$ whose density $p_\theta(t)$ has monotone likelihood ratio in $t$. Fix $\alpha\in(0,1)$. Suppose that for each $\theta_0\in\Theta$ there exists a strictly unbiased UMP unbiased level-$\alpha$ test of $H(\theta_0):\theta=\theta_0$ whose acceptance region is of the form
\[
\mathcal A_{\theta_0,\alpha}=\{t:\ a_\alpha(\theta_0)\le t\le b_\alpha(\theta_0)\},
\]
with $a_\alpha(\cdot)$ and $b_\alpha(\cdot)$ strictly increasing, so that the inverted confidence set
\[
C_\alpha^\star(X)=\{\theta\in\Theta:\ T(X)\in \mathcal A_{\theta,\alpha}\}
=\bigl[b_\alpha^{-1}\{T(X)\},\ a_\alpha^{-1}\{T(X)\}\bigr]
\]
is a most accurate unbiased $(1-\alpha)$ confidence interval for $\theta$. Assume moreover that \[\{C_\alpha^\star(X):\alpha\in(0,1)\}\] can be chosen nested in $\alpha$ and that $E_\theta|C_\alpha^\star(X)|<\infty$ for all $\theta,\alpha$. Define the induced possibility contour by stacking
\[
\pi_X^\star(\vartheta)=\sup\{\alpha\in[0,1]:\ \vartheta\in C_\alpha^\star(X)\},\qquad \vartheta\in\Theta,
\]
and let $\Pi_X^\star$ be the corresponding possibility measure.

Then $\Pi_X^\star$ is valid and possibilistically unbiased. Furthermore, for the concentration penalty $\Gamma_x(\vartheta)$ in \eqref{eq:concentration-penalty} with $\phi(\ell)=\ell$ (integrated expected length), the procedure $\Pi^\star=\{\Pi_X^\star\}$ minimizes Choquet risk pointwise in $\theta$ over the class of valid, possibilistically unbiased procedures whose $\alpha$-cuts are unbiased $(1-\alpha)$ confidence intervals with finite expected length. That is, $\Pi^\star$ is concentration-optimal among unbiased valid contours.
\end{corollary}

\begin{proof}
Validity and possibilistic unbiasedness follow from Proposition~\ref{prop:unbiased-equiv} applied to the nested family $\{C_\alpha^\star(X)\}$. 
For the stated concentration penalty, Corollary~\ref{cor:concentration-optimal} reduces Choquet risk to an integral of expected lengths of $\alpha$-cuts. 
Since $C_\alpha^\star(X)$ has minimal expected length among unbiased $(1-\alpha)$ intervals at each $\alpha$ (by the Ghosh-Pratt identity), integrating over $\alpha$ yields pointwise minimal Choquet risk within the stated class.
\end{proof}

\begin{example}[Exponential model: unbiased concentration-optimal contours]\label{ex:expunb}
Let $X_1,\ldots,X_n$ be i.i.d.\ exponential with rate $\lambda>0$,
\[
f_\lambda(x)=\lambda e^{-\lambda x}\,1\{x\ge 0\},\qquad \lambda\in\Theta=(0,\infty),
\]
and write $S=\sum_{i=1}^n X_i$ and $\bar X=S/n$. The likelihood is
\[
L_x(\lambda)=\lambda^n e^{-\lambda S},\qquad \hat\lambda=\frac{n}{S}=\frac{1}{\bar X}.
\]
Consider the raw relative likelihood ratio contour
\begin{equation}\label{eq:exp-q}
q_x(\lambda)
=\frac{L_x(\lambda)}{L_x(\hat\lambda)}
=\Big(\frac{\lambda}{\hat\lambda}\Big)^n\exp\{-(\lambda-\hat\lambda)S\}
=(\lambda\bar X)^n\exp\{-n(\lambda\bar X-1)\}.
\end{equation}
Validify $q$ via the probability integral transform (PIT) and define the calibrated LR contour
\begin{equation}\label{eq:exp-pi-unb}
\pi_x^{\mathrm{unb}}(\lambda)
:=P_\lambda\{q_X(\lambda)\le q_x(\lambda)\},\qquad \lambda\in(0,\infty).
\end{equation}
For each fixed $\lambda_0$, continuity implies $\pi_X^{\mathrm{unb}}(\lambda_0)\sim \mathrm{Unif}(0,1)$ under $P_{\lambda_0}$, so $\pi^{\mathrm{unb}}$ is \emph{exactly valid}. Moreover, $\pi_x^{\mathrm{unb}}(\lambda_0)$ is the exact finite-sample two-sided LR $p$-value for testing $H_0:\lambda=\lambda_0$.

It remains to identify this exact LR p-value as the UMPU p-value. The statistic
\(S\) is sufficient and the family
\[
  f_\lambda(s)
  =
  \frac{\lambda^n}{\Gamma(n)}s^{n-1}e^{-\lambda s},
  \qquad s>0,
\]
has monotone likelihood ratio (increasing MLR in $-S$). Fix \(\lambda_0\) and
write \(Y=\lambda_0S\). Under \(H_0\), \(Y\sim\Gamma(n,1)\), and
\[
  q_X(\lambda_0)
  =
  \left(\frac{Y}{n}\right)^n e^{-(Y-n)}
  \propto
  Y^n e^{-Y}.
\]
Since \(y\mapsto y^n e^{-y}\) is unimodal with maximizer \(n\), the LR
rejection region has the form
\[
  \{Y\le y_{1,\alpha}\}\cup\{Y\ge y_{2,\alpha}\},
  \qquad y_{1,\alpha}<n<y_{2,\alpha},
\]
where the cutoffs are chosen so that the rejection probability is \(\alpha\)
under \(Y\sim\Gamma(n,1)\) and the two boundary likelihoods are equal:
\[
  y_{1,\alpha}^n e^{-y_{1,\alpha}}
  =
  y_{2,\alpha}^n e^{-y_{2,\alpha}}.
\]
In this one-parameter exponential-family setting, the UMP unbiased level-\(\alpha\)
test of \(H_0:\lambda=\lambda_0\) is characterized by the level condition together
with the Lehmann--Romano moment condition
\cite[Sec.~4.2, eqs.~(4.5)--(4.6)]{lehmann2005testing}. In the present gamma
case, the equal-likelihood boundary condition above is equivalent, by a single
integration-by-parts identity, to that moment condition. Therefore the exact
two-sided LR test is UMPU, and \(\pi_x^{\rm unb}\) is the corresponding UMPU
p-value function.

The \(\alpha\)-cuts of this contour are
\[
  C_\alpha(X)
  :=
  \{\lambda:\pi_X^{\rm unb}(\lambda)\ge \alpha\}
  =
  \left[
    \frac{y_{1,\alpha}}{S},
    \frac{y_{2,\alpha}}{S}
  \right].
\]
These are the most accurate unbiased
\((1-\alpha)\) confidence intervals for \(\lambda\). They are nested in \(\alpha\), and
for \(n>1\), these intervals have finite expected length, since
\(|C_\alpha(X)|=(y_{2,\alpha}-y_{1,\alpha})/S\) and
\(E_\lambda(S^{-1})=\lambda/(n-1)\).
Therefore, by Corollary \ref{cor:optimal-unbiased}, the induced possibility measure is
concentration-optimal, for the integrated expected-length penalty, within the class
of valid possibilistically unbiased procedures.
\end{example}

\subsection{Equivariance and Minimax Choquet Risk}\label{sec:equi}
The previous section showed how classical unbiasedness theory can be lifted into the possibilistic setting. That route is powerful, but it is also very special. Outside highly structured settings, one should not expect a procedure that is uniformly best at every parameter value. This is already familiar from classical decision theory: pointwise optimality is rare, and the usual replacement is a minimax criterion instead. The minimax criterion, however, is most useful when combined with structural symmetries. In invariant statistical problems such as translation models, scale models, translation-scale models, and other group models, equivariance properties are used to restrict attention to procedures that respect the relevant symmetries of the statistical problem. Hunt-Stein type arguments then show that this restriction entails no loss of generality for minimax comparisons.

The aim of this section is to develop the corresponding statment for valid possibility contours. The key observation is that the Choquet-risk decomposition already puts us in the familiar classical territory: a valid contour is equivalently a nested family of calibrated confidence sets, and Choquet risk is an integral over the risks of those $\alpha$-cuts. Thus equivariance can be imposed either on the contour itself or on every $\alpha$-cut. Once that equivalence is established, classical invariant confidence-set arguments can be applied level by level and then integrated over $\alpha$.

\subsubsection{Equivariant contours and equivariant \texorpdfstring{$\alpha$}{alpha}-cuts}
Let $G$ be a group acting measurably on the sample space $\mathcal X$ and on the parameter space $\Theta$. We write these actions as \[x\mapsto gx, \qquad \theta \mapsto g\theta.\] The model $\{P_\theta:\theta \in \Theta \}$ is $G$-invariant if \[X\sim P_\theta \implies gX\sim P_{g\theta}, \qquad g\in G, \theta \in \Theta.\] In such a problem, transforming both the data and the parameter should not change the inferential content of the contour. It should merely relabel it. This leads to the following definition.

\begin{definition}[Equivariant possibility]\leavevmode\par
A data-dependent possibility contour $x\mapsto \pi_x$ is $G$-equivariant if \[\pi_{gx}(g\theta)=\pi_x(\theta), \qquad g\in G, x\in \mathcal X, \theta \in \Theta.\] Equivalently, the induced possibility measure satisfies \[\Pi_{gx}(gB)=\Pi_x(B), \qquad B\subseteq\Theta,\] where $gB=\{g\theta:\theta \in B\}.$ 
\end{definition}

This definition is the natural possibilistic analogue of equivariance for confidence sets. Recall that the $\alpha$-cut of the contour is \[A_\alpha(x)=\{\theta:\pi_x(\theta)\geq \alpha \}.\] The next lemma records the basic equivalence: imposing equivariance on the contour is the same as imposing equivariance on every $\alpha$-cut. This is a bookkeeping result that lets us translate between possibility-contour language and the classical theory of equivariant confidence regions.

\begin{lemma}\label{lem: equiv**2}
    Let $\pi_x$ be a valid possibility contour with $\alpha$-cuts \[A_\alpha(x)=\left\{\theta: \pi_x(\theta) \geq \alpha \right\},\] and suppose conversely that the contour is recovered from its $\alpha$-cuts by
    \begin{equation*}
        \pi_x(\theta)=\sup\left\{\alpha\in[0,1]:\theta \in A_\alpha(x)\right\}.
    \end{equation*}
    Then the following are equivalent:
    \begin{enumerate}
  \item $\pi_x$ is valid and $G$-equivariant;
\item  $A_\alpha$ is a $G$-equivariant
$(1-\alpha)$-confidence set, i.e. \[A_\alpha(gx)=gA_\alpha(x), \qquad P_\theta\{\theta \in A_\alpha(X)\}\geq 1-\alpha, \quad \text{for all $\alpha \in (0,1)$}.\]
  
\end{enumerate}
\end{lemma}
\begin{proof}
The equivalence of validity and nominal coverage of the $\alpha$-cuts is shown in Proposition \ref{prop:unbiased-equiv}. What remains is the bookkeeping for the equivariance conditions.

    (1$\implies$2): Take any $\alpha\in(0,1)$, then for any $\theta$
    \begin{align*}
        \theta \in A_\alpha(g  x) &\Leftrightarrow \pi_{gx}(\theta)\geq \alpha \Leftrightarrow \pi_x(g^{-1} \theta) \geq \alpha \\
        &\Leftrightarrow g^{-1}\theta \in A_\alpha(x) \Leftrightarrow\theta \in g  A_\alpha(x),
    \end{align*}
    where the last equivalence uses the bijectivity of $g$. Hence \[A_\alpha(gx)=gA_\alpha(x).\] \\
    (2$\implies$1): Conversely, suppose the $\alpha$-cuts are equivariant. Then, for any $g,x,\theta$,
    \begin{equation*}
        \pi_{gx}(g\theta)=\sup\left\{\alpha:g \theta \in A_\alpha(g  x) \right\}= \sup\left\{\alpha: \theta \in A_\alpha(x)\right\}=\pi_x(\theta),
    \end{equation*}
    where the second equality again follows from the bijectivity of $g$. Hence the contour is $G$-equivariant.

\end{proof}

The importance of Lemma \ref{lem: equiv**2} is that it makes equivariance compatible with the $\alpha$-cut reduction of Choquet risk. If a loss is invariant under the group action, then the problem of minimizing Choquet risk over valid equivariant contours can be studied level by level as a problem about equivariant confidence sets.

\subsubsection{A Hunt-Stein reduction for Choquet risk}
We now turn from definition to minimax comparisons. Let $\Pi_v$ denote a class of valid possiblity contours, and let $\Pi_v^e$ denote the subclass of valid $G$-equivariant contours. Assume the loss-penalty is invariant in the sense that \[\Gamma_{g\theta}(g\vartheta)=\Gamma_\theta(\vartheta), \qquad g\in G.\] The most important example for this section are \emph{radial penalties}, of form \[\Gamma_\theta(\vartheta)=h(\rho(\vartheta, \theta)),\] where $\rho$ is a $G$-invariant distance-like function, \[\rho(g\vartheta, g\theta)=\rho(\vartheta, \theta),\] and $h$ nondecreasing.

For a contour $\pi$, write \[r_\alpha(\theta,\pi)=E_\theta \bigg[\sup_{\vartheta \in A_\alpha(X)}\Gamma_\theta(\vartheta)\bigg],\] and \[R(\theta, \pi)=\int_0^1 r_\alpha(\theta, \pi)d\alpha.\]

The next result is the possibilistic analogue of the usual Hunt-Stein reduction. It says that, for minimax purposes, non-equivariant valid contours can be symmetrized (made equivariant in distribution) without increasing worst-case Choquet risk. As a consequence, for minimax comparisions of contours, we can restrict attention to only equivariant contours.

\begin{theorem}\label{thm:H-S}
    Assume that $G$ is a compact or amenable group and acts measurably on $(\mathcal X, \Theta)$, and the model $\{P_\theta:\theta \in \Theta\}$ is $G$-invariant. Suppose the loss-penalty is radial and $G$-invariant.
   Then, for every valid contour $\pi\in\Pi_v$, there exists a randomized $G$-equivariant (in distribution) valid contour $\pi^{(\mu)}$ such that \[\sup_{\theta \in \Theta} R(\theta, \pi^{(\mu)}) \leq \sup_{\theta \in \Theta} R(\theta, \pi).\] Consequently, the minimax value over all valid contours equals the minimax value over randomized equivariant valid contours:
        \begin{equation*}
           \inf_{\pi\in \Pi_v} \sup_\theta R(\theta, \pi)= \inf_{\pi \in \Pi^e_v} \sup_\theta R(\theta, \pi).
        \end{equation*}
  For noncompact amenable groups, the same conclusion holds by replacing Haar averaging with a Følner/Hunt-Stein limiting argument
\end{theorem}
\begin{proof}

    For any valid $\pi\in \Pi_v$, average its contour over $G$ to obtain the randomized symmetrization $\pi^{(\mu)}$. In the compact case, Lemma \ref{lem:A1} shows: (i) $\pi^{(\mu)}$ is valid; (ii) it is $G$-equivariant in distribution; (iii) and, crucially,
    \begin{equation}\label{eq:Haaravg}
        r_\alpha\bigl(\theta, \pi^{(\mu)}\bigr) = \int_G r_\alpha(g\cdot \theta, \pi)\mu(dg) \quad \implies \quad \sup_\theta r_\alpha\bigl(\theta, \pi^{(\mu)}\bigr) \leq \sup_\theta r_\alpha(\theta, \pi),
    \end{equation}
    hence $\sup_\theta R\bigl(\theta, \pi^{(\mu)}\bigr) \leq \sup_\theta R(\theta, \pi)$. Therefore,
    \begin{equation}\label{eq: minmaxrisk}
        \inf_{\pi \in \Pi_v} \sup_\theta R(\theta, \pi) =  \inf_{\pi \in \Pi_v^e} \sup_\theta R(\theta, \pi).
    \end{equation}


In the non-compact amenable group $G$ case, we replace the Haar average in \eqref{eq:Haaravg} by a Følner sequence and invoke the standard Hunt-Stein limiting argument (see remark \ref{rmk:amenable}) to obtain an exactly invariant randomized procedure with the same minimax bound.

\end{proof}

The theorem should be read as a reduction principle. It tells us that minimax optimality can be sought among equivariant procedures, but it does not by itself identify the minimax contour. To obtain an actual optimizer, we still need to solve the levelwise equivariant confidence-set problem. In general this can be difficult. The cleanest solvable case is the Gaussian location experiment.

\subsubsection{Gaussian location: a radial minimax contour}\label{sec:Gaussball}
Consider the Gaussian location experiment \[Z \sim N_d(u, I_d), \qquad u \in \mathbb R^d,\] with the translation group $G=(\mathbb R^d, +)$ acting by \[z \mapsto z+t, \qquad u\mapsto u+t.\] Let the penalty be radial, \[\Gamma_u(v)=h(\|v-u\|),\] where $h$ is nondecreasing and satisfies the mild integrability assumptions needed below.

The candidate optimal contour is the Gaussian possibility contour \[\pi_z^\star(v)=1-F_d(\|v-z\|^2),\] where $F_d$ is the distribution function of $\chi^2_d$. Its $\alpha$-cuts are the usual Gaussian confidence balls \[A^\star_\alpha(z)=\{v:\|v-z\|^2\leq F_d^{-1}(1-\alpha)\}.\] These are preciisely the translation-equivariant confidence regions one would expect from the ordinary $z$-test. The point of the next result is that they are also optimal for Choquet risk with respect to the radial penalty.

\begin{corollary}
\label{cor:gaussian-possibility-minimax}
Let $\Pi_v^{\mathrm{tr}}$ be the class of measurable valid
translation-equivariant possibility contours in the Gaussian location experiment \[Z\sim N_d(u, I_d).\] Let \[\Gamma_u(v)=h(\|v-u\|)\] with $h$ continuous, nondecreasing, $h(r)\rightarrow\infty$, and $E[h(r+\|W\|)]<\infty$ for every $r<\infty$, where $W\sim N_d(0, I_d)$. Define \[\pi_z^\star(v)=1-F_d(\|v-z\|^2).\] Then $\pi^\star$ is valid and translation-equivariant. Moreover, \[R(u, \pi^\star)\leq R(u, \pi)\qquad \text{for every $u\in \mathbb R^d$ and every $\pi\in \Pi_v^{tr}$.}\] Consequently, $\pi^\star$ is Choquet-risk minimal within the class of valid translation-equivariant contours. By the Hunt-Stein reduction, it is minimax over the full class of valid contours: \[\sup_u R(u, \pi^\star)=\inf_{\pi\in\Pi_v}\sup_u R(u, \pi).\]
\end{corollary}

\begin{proof}
First, $\pi^\star$ is a normalized possibility contour, since
\[
\sup_{v\in\mathbb{R}^d}\pi_z^\star(v)=\pi_z^\star(z)=1.
\]
It is translation-equivariant because
\[
\pi_{z+t}^\star(v+t)
=
1-F_d(\|v+t-(z+t)\|^2)
=
1-F_d(\|v-z\|^2)
=
\pi_z^\star(v).
\]
It is also valid: under $P_u$,
\[
\pi_Z^\star(u)=1-F_d(\|u-Z\|^2),
\]
and $\|u-Z\|^2\sim \chi_d^2$, so by continuity of $F_d$,
\[
\pi_Z^\star(u)\sim \mathrm{Unif}(0,1).
\]
Hence $\pi^\star\in \Pi_v^{\mathrm{tr}}$.

For $\alpha\in(0,1)$,
\[
A_\alpha^{\pi^\star}(z)
=
\{v:1-F_d(\|v-z\|^2)\ge \alpha\}
=
\{v:\|v-z\|^2\le F_d^{-1}(1-\alpha)\},
\]
which is exactly the ball rule $C_\alpha^\star(z)$ from
Proposition~\ref{prop:gaussian-shift-balls}.

Now let $\pi\in \Pi_v^{\mathrm{tr}}$. By Lemma \ref{lem: equiv**2}, for each $\alpha\in(0,1)$
the $\alpha$-cut $A_\alpha^\pi$ is a measurable translation-equivariant
$(1-\alpha)$-confidence set, i.e.
\[
A_\alpha^\pi \in C_\alpha^{\mathrm{tr}}.
\]
Therefore Proposition~\ref{prop:gaussian-shift-balls} gives
\[
r_\alpha\bigl(u,A_\alpha^{\pi^\star}\bigr)
\le
r_\alpha\bigl(u,A_\alpha^\pi\bigr)
\qquad
\forall \alpha\in(0,1),\ \forall u\in\mathbb{R}^d.
\]
Integrating in $\alpha$ and applying Lemma \ref{lem:alpha-cut-minimal} yields
\[
R(u,\pi^\star)\le R(u,\pi)
\qquad
\forall u\in\mathbb{R}^d,\ \forall \pi\in \Pi_v^{\mathrm{tr}}.
\]
Thus $\pi^\star$ is Choquet-risk minimal at every $u$ over
$\Pi_v^{\mathrm{tr}}$.

Finally, Since the acting group $G=(\mathbb R^d, +)$ is non-compact but amenable, the reduction in Theorem \ref{thm:H-S} can be applied, and we have
\[
\inf_{\pi\in\Pi_v}\sup_u R(u,\pi)
=
\inf_{\pi\in\Pi_v^{\mathrm{tr}}}\sup_u R(u,\pi).
\]
Since $\pi^\star\in \Pi_v^{\mathrm{tr}}$ attains the right-hand side, it also
attains the left-hand side. Hence $\pi^\star$ is minimax over the full class
$\Pi_v$ of valid possibility contours.

\end{proof}

\subsubsection{Conditional validification in transitive models}
The Gaussian location result suggests a broader recipe. The proof of Proposition \ref{prop:gaussian-shift-balls} relies heavily on the interplay between the functional form of the $\alpha$-cut risk functional \[r_\alpha(u, \pi)=E_\theta\bigg[\sup_{v \in A_\alpha(X)} h(\|v-u\|)\bigg]\] and convex geometry. The risk only cares about the sup, not the shape, which allows for a reduction to ball sets around an equivariant center, with radius determined by the radial penalty. We suspect similar reductions are available in more complex geometries than the earlier location model, though we are yet to find any. The following discussion generalizes the construction of the contour in Corollary \ref{cor:gaussian-possibility-minimax} — though unless a similar ball reduction is available, no statement of optimality applies.

In a transitive invariant model, suppose we have an equivariant center $\hat \theta(X)$ and an invariant distance-like function $\rho$. A natural family of equivariant confidence sets is then given by balls of the form \[\{\vartheta:\rho(\vartheta, \hat \theta(X))\leq r\}.\] The remaining question is how to choose the radius. If no ancillary or orbit information is available, the radius may be fixed by ordinary validification. But in many invariant models there is a maximal invariant $M(X)$ that carries ancillary information relevant to the radius. Conditioning on $M(X)$ can therefore give sharper valid regions.

This motivates the following strengthened notion of validity.
\begin{definition}[$M$-conditionally valid confidence set]\label{def:msim}
Let $M(X)$ be a maximal invariant. A confidence set $C_\alpha(X)$ is $M$-conditionally valid if \[P_\theta\{\theta \in C_\alpha(X)|M(X)\}\geq 1-\alpha\qquad \text{for every $\theta \in \Theta$, $P_\theta$-a.s.}\]  A possibility contour is $M$-conditionally valid if every $\alpha$-cut is $M$-conditionally valid.
\end{definition}

By the tower property, $M$-conditional validity implies ordinary validity. It is stronger than what is required for calibration, but it is often natural in invariant problems because the maximal invariant is ancillary to the parameter of interest, or because classical invariant procedures are already constructed conditionally on such a statistic.

We now describe the corresponding radius-validified contour. We reiterate: this is not, by itself, a general minimax theorem. Rather, it is a construction of the smallest conditionally valid equivariant ball contour, conditional on the chosen center, invariant distance, and maximal invariant. It becomes a minimax contour in settings where one can additionally prove that it is enough to consider such ball rules. 


\begin{proposition}\label{prop:validifiedradius}
    Assume the invariant setup of Lemma \ref{lem:A3}. In particular, let $G$ act transitively on $\Theta$, let the model be $G$-invariant, let $\hat \theta$ be a $G$-equivariant center, let $M$ be a maximal $G$-invariant, and let $\rho$ be a measurable $G$-invariant distance-like function satisfying $\rho(\theta, \theta)=0$. Fix a reference $\theta_0\in \Theta$, and define \[H(m,t):=P_{\theta_0}\{\rho(\theta_0,\hat \theta(X))\leq t|M(X)=m\}, \qquad m\in \mathcal M, t\geq 0.\] Assume $H$ satisfies the regularity conditions in Lemma \ref{lem:A3}, and define \[\tau^*_\alpha:=\inf\{t\geq 0:H(m,t)\geq 1-\alpha\}, \qquad \alpha \in (0,1).\] For $t\geq 0$, write \[H(m, t-)=\lim_{s \uparrow t} H(m,s), \] with $H(m, 0-)=0$. Define the contour directly 
    \begin{equation}\label{eq:radiusdirect}
        \pi_x^*(\vartheta)=P_{\theta_0}\{\rho(\theta_0, \hat \theta(X))\geq \rho(\vartheta, \hat \theta(x))|M(x)=m(x)\}=1-H\big(M(x), \rho(\vartheta, \hat \theta(x))-\big).
    \end{equation}

    Assume the mild quantile-separation condition \[t > \tau^*_\alpha(m) \implies H(m, t-)>1-\alpha \quad \text{for $\mu_M$-a.e. $m$} .\]  This is automatic e.g. if $H(m, \cdot)$ is continuous and has no flat portions. Then $\pi^*$ is a normalized, $G$-equivariant, $M$-valid (hence valid) possibility contour. Its $\alpha$-cuts are exactly the orbitwise minimal feasible-radius confidence sets \[A_\alpha^*(x)=\{\vartheta:\pi_x^*(\vartheta)\geq \alpha \}=\vartheta:\rho(\vartheta, \hat \theta(x))\leq \tau^*_\alpha(M(x))\}.\] Consequently, the $\alpha$-cuts of $\pi^*$ are nested, valid, $G$-equivariant, and orbitwise minimal among $M$-dependent $\rho$-ball procedures satisfying the conditional coverage constraint. Thus, for any Choquet-risk criterion whose $\alpha$-cut risk is monotone in these radii, $\pi^*$ is the corresponding orbitwise radius-minimal contour.
\end{proposition}
\begin{proof}
    The detailed proof of measurability, equivariance, validity, nesting, and orbitwise minimality for the confidence sets $A^*_\alpha(x)$ is given in Lemma \ref{lem:A3}. The only thing not contained in the proof of Lemma \ref{lem:A3} is that the direct construction given above in \eqref{eq:radiusdirect} has $\alpha$-cuts given by these confidence sets $A^*_\alpha(x)$.

    Fix $x,m=M(x)$, and write \[r=\rho(\vartheta,\hat \theta(x)).\] Then \[\pi_x^*(\vartheta)\geq \alpha \Longleftrightarrow 1-H(m, r-)\geq\alpha \Longleftrightarrow H(m, r-)\leq 1-\alpha.\] By the generalized-inverse definition of $\tau^*_\alpha(m)$, if $r<\tau^*_\alpha(m)$, then $H(m, r-)\leq 1-\alpha$. Conversely, under the quantile-separation condition, if $r>\tau^*_\alpha(m)$, then $H(m,r-)>1-\alpha$. Hence \[\pi_x^*(\vartheta\geq\alpha \Longleftrightarrow r\leq \tau^*_\alpha(m),\] which gives \[\{\vartheta:\pi_x^*(\vartheta)\geq \alpha\}=\{\vartheta: \rho(\vartheta, \hat \theta(x))\leq \tau^*_\alpha(M(x))\}.\] That is exactly the ball family identified in Lemma \ref{lem:A3}.
\end{proof}

Proposition \ref{prop:validifiedradius} gives the general conditional-validification recipe. In the Gaussian location model, with trivial maximal invariant, equivariant center $\hat u(Z)=Z$, and squared Euclidean radius $\rho(v,z)=\|v-z\|^2$, this recipe yields the Gaussian contour in Corollary \ref{cor:gaussian-possibility-minimax}. And in the Gaussian location model, there is available a ball-reduction result, so the contour is also minimax.

\section{Direct Comparisons via Simulation Experiments}\label{sec:directComp}
Most of the paper so far has focused on \emph{theoretical} comparisons of possibilistic procedures. Those results are valuable because they identify ``best possible'' procedures within a principled class, and explain \emph{why} optimality emerges.

But the Choquet-risk framework is also useful in a more pragmatic, \emph{case-by-case} setting. It can be used to directly compare two (or more) concrete procedures by estimating their pointwise Choquet risks via Monte Carlo and plotting resulting risk surfaces. This complements the inferential model (IM) literature, where comparisons are often presented as side-by-side contour plots for a particular dataset. For example, \citet{Martin2025} emphasizes nuisance-parameter elimination as a key practical issue and compares \emph{extension-based} and \emph{profile-based} marginalization strategies, including normal-mean and gamma-mean illustrations.

The key hindrance for direct comparisons is primarily \emph{computational cost}, as the Monte Carlo simulations can quickly become expensive. Another major hindrance is the fact that we have to rely on visual inspection. For multiparameter comparisons, or marginal comparisons where risk depends on nuisance, we have to either limit the scope of the comparison or introduce yet more subjective choice by e.g. integrating out parameter influence via priors, or taking averages—also further exaggerating the computational burden.

Details on computational implementation for these simulation comparison studies will not be covered here, but some quick comments are in order. For more on the computational aspects of IMs, see \citet{MartinIMMC}. First, because the $\alpha$-integral in the Choquet risk representation $\int_0^1r_\alpha(\theta,\Pi)d\alpha$ is with respect to the uniform measure, we can ``randomize $\alpha$''and estimate the risk via a doubly-randomized Monte Carlo scheme. If $U\sim U(0,1)$ is independent of $X$, then
\begin{equation*}
    R(\theta,\Pi)=E_\theta E_U\bigl[\sup_{\vartheta \in A_U(X)} \Gamma(X, \theta. \vartheta) \bigr],
\end{equation*}
so a single (or small number of) $\alpha$-levels per dataset yields an unbiased risk estimate. This aids computational implementation, but not efficiency; since the integral is one-dimensional, the Monte Carlo error is $O(1/\sqrt{n})$. Deterministic methods typically converge faster. Second, when multiple $\alpha$-cuts are needed per data-draw, nestedness can often be exploited. Evaluating levels in decreasing (or increasing) order allows warm-starting the inner optimization, so the marginal cost of additional $\alpha$'s is typically far smaller than recomputing each optimization from scratch. Thirdly, the inner supremum often admits structural shortcuts. For many distance-to-truth penalties, maximizers over $A_\alpha(x)$ occur on the boundary level set $\{\pi_x=\alpha\}$, which motivates boundary-based optimization/sampling rather than searching over the full cut. Finally, since for direct comparisons it is often the \emph{relative risk} that is of primary interest, reusing common random numbers for the risk computations per procedure can drastically reduce the Monte-Carlo noise in the relative risk estimate.

In some structured comparisons, the direct risk comparison is simpler than a
full risk surface.  If the model, contours, and penalty are compatible with the
same group structure, then the pointwise risks may either be parameter-free, or
share a common scale factor.  In those cases the relative risk reduces to a single
number for each sample size.  The next example illustrates this in the exponential
rate model.
\begin{example}[Exponential rate: radial contour vs. calibrated LR]
Let $X_1,\ldots,X_n$ be i.i.d. exponential with rate $\lambda>0$, and write
\[
        S=\sum_{i=1}^n X_i,\qquad \hat\lambda=\frac{n}{S}.
\]
We compare two valid contours for $\lambda$.  The first is the calibrated
likelihood-ratio contour from Example \ref{ex:expunb},
\[
        \pi_x^{\rm LR}(\lambda)
        =
        P_\lambda\{q_X(\lambda)\le q_x(\lambda)\},
\]
where
\[
        q_x(\lambda)
        =
        \left(\frac{\lambda S}{n}\right)^n
        \exp\{-(\lambda S-n)\}.
\]
Equivalently, if $Y\sim \Gamma(n,1)$, then the $\alpha$-cut of
$\pi^{\rm LR}$ is
\[
        A_\alpha^{\rm LR}(x)
        =
        \left[
          \frac{y_{1,\alpha}}{S},
          \frac{y_{2,\alpha}}{S}
        \right],
        \qquad \alpha\in(0,1),
\]
where $y_{1,\alpha}<n<y_{2,\alpha}$ are chosen so that
\[
        P\{y_{1,\alpha}\le Y\le y_{2,\alpha}\}=1-\alpha,
        \qquad
        y_{1,\alpha}^n e^{-y_{1,\alpha}}
        =
        y_{2,\alpha}^n e^{-y_{2,\alpha}}.
\]
The equal-likelihood endpoint condition is also the usual unbiasedness
condition for the inverted two-sided test, since
\[
        \frac{d}{dy}\{y^n e^{-y}\}
        =
        -(y-n)y^{n-1}e^{-y}.
\]
Thus $\pi^{\rm LR}$ is the exact two-sided LR, equivalently UMPU, p-value
contour.

The second contour is the radial contour obtained from Proposition
\ref{prop:validifiedradius}.  The scale group acts by
\[
        c\cdot x = x/c,\qquad c\cdot \lambda=c\lambda,
        \qquad c>0.
\]
The estimator $\hat\lambda=n/S$ is equivariant, and a natural invariant radius is
the log-radius
\[
        \rho(\lambda,\lambda')=|\log(\lambda/\lambda')|.
\]
The maximal invariant for this scale action is, for example,
\[
        M(X)=\left(\frac{X_1}{S},\ldots,\frac{X_n}{S}\right),
\]
viewed on the simplex.  In the exponential scale model, $M(X)$ is independent
of $S$.  Therefore the conditional radius distribution required in Proposition
\ref{prop:validifiedradius} is the same as the unconditional one.  Since
$\lambda S\sim\Gamma(n,1)$,
\[
        \rho(\lambda,\hat\lambda)
        =
        \left|\log\frac{\lambda S}{n}\right|
\]
has a distribution that does not depend on $\lambda$.  Hence the
radial-validified contour is
\[
        \pi_x^{\rm rad}(\vartheta)
        =
        P\left\{
          \left|\log\frac{n}{Y}\right|
          \ge
          \left|\log\frac{\vartheta}{\hat\lambda}\right|
        \right\},
        \qquad Y\sim\Gamma(n,1).
\]
Its $\alpha$-cuts are log-symmetric intervals around $\hat\lambda$:
\[
        A_\alpha^{\rm rad}(x)
        =
        \left[
          \hat\lambda e^{-c_{n,\alpha}},
          \hat\lambda e^{c_{n,\alpha}}
        \right],
        \qquad \alpha\in(0,1),
\]
where $c_{n,\alpha}$ is defined by
\[
        P\left\{
          \left|\log\frac{n}{Y}\right|
          \le c_{n,\alpha}
        \right\}
        =
        1-\alpha,
        \qquad Y\sim\Gamma(n,1).
\]
Thus the two procedures differ in finite samples: the LR contour uses
equal-likelihood endpoints, while the radial contour enforces log-symmetric
balls around the equivariant center.

We compare the procedures under three penalties.  The first is log-radius
accuracy,
\[
        \Gamma_\lambda^{\log}(\vartheta)
        =
        |\log(\vartheta/\lambda)| .
\]
The second is ordinary Lebesgue concentration on the rate scale, i.e. the
concentration penalty in \eqref{eq:concentration-penalty} with
$\varphi(\ell)=\ell$, so that the loss is
\[
        L^{\rm Leb}(\lambda,\Pi_x)
        =
        \int_0^1 |A_\alpha(x)|\,d\alpha .
\]
The third is per-observation symmetrized Kullback--Leibler accuracy,
\[
        \Gamma_\lambda^{\rm sKL}(\vartheta)
        =
        {\rm KL}(P_\lambda,P_\vartheta)
        +
        {\rm KL}(P_\vartheta,P_\lambda)
        =
        \frac{\vartheta}{\lambda}
        +
        \frac{\lambda}{\vartheta}
        -2 .
\]
If $P_\lambda$ is instead taken to denote the full $n$-sample law, then the
right-hand side is multiplied by $n$; this common factor does not affect the
relative-risk comparisons below.

For both contours, the $\alpha$-cuts have the form
\[
        A_\alpha^m(X)=\hat\lambda\,B_{\alpha,n}^m,
        \qquad m\in\{\mathrm{LR},\mathrm{rad}\},
\]
where the deterministic sets $B_{\alpha,n}^m$ are
\[
        B_{\alpha,n}^{\rm LR}
        =
        \left[
          \frac{y_{1,\alpha}}{n},
          \frac{y_{2,\alpha}}{n}
        \right],
        \qquad
        B_{\alpha,n}^{\rm rad}
        =
        \left[
          e^{-c_{n,\alpha}},
          e^{c_{n,\alpha}}
        \right].
\]
Since
\[
        \frac{\hat\lambda}{\lambda}
        =
        \frac{n}{\lambda S},
        \qquad \lambda S\sim\Gamma(n,1),
\]
the ratio $\hat\lambda/\lambda$ has a distribution depending only on $n$.
Therefore the log-radius risks do not depend on $\lambda$:
\[
        R_n^{\log}(\lambda,\Pi^m)=R_n^{\log}(1,\Pi^m).
\]
The same scale-invariance argument gives
\[
        R_n^{\rm sKL}(\lambda,\Pi^m)=R_n^{\rm sKL}(1,\Pi^m),
        \qquad n>1,
\]
where the condition $n>1$ ensures finiteness, since the sKL loss contains a
term proportional to $\hat\lambda/\lambda=n/(\lambda S)$ and
$E\{1/(\lambda S)\}<\infty$ only for $n>1$.

For the Lebesgue concentration penalty, lengths scale linearly with the rate:
\[
        |A_\alpha^m(X)|
        =
        \hat\lambda\,|B_{\alpha,n}^m|
        =
        \lambda
        \left(\frac{\hat\lambda}{\lambda}\right)
        |B_{\alpha,n}^m|.
\]
Consequently
\[
        R_n^{\rm Leb}(\lambda,\Pi^m)
        =
        \lambda R_n^{\rm Leb}(1,\Pi^m),
        \qquad n>1,
\]
again with $n>1$ ensuring finiteness.  Thus, for finite-risk comparisons under
all three penalties, the relative risks
\[
        {\rm RR}_n^\Gamma
        =
        \frac{R_n^\Gamma(\lambda,\Pi^{\rm rad})}
             {R_n^\Gamma(\lambda,\Pi^{\rm LR})}
\]
are independent of $\lambda$.  Hence, for each sample size $n$, the comparison
reduces to a single dimensionless number.

Figure~\ref{fig:exp-rate-relative-risk} plots these relative risks as $n$ varies.
Values above one indicate that the radial contour has larger Choquet risk than
the calibrated LR contour.  The LR contour is uniformly preferred in this
experiment under all three criteria.  The finite-sample cost of enforcing radial
log-symmetry is largest for Lebesgue concentration, moderate for symmetrized
KL accuracy, and smallest for log-radius accuracy.  In each case the relative
risk approaches one as $n$ increases. Since the LR contour converges to a Gaussian contour, this might be reflecting that both contours approach the
same local Gaussian/Fisher--Rao shape asymptotically.
\end{example}

\begin{figure}[H]
    \centering
    \includegraphics[width=0.5\linewidth]{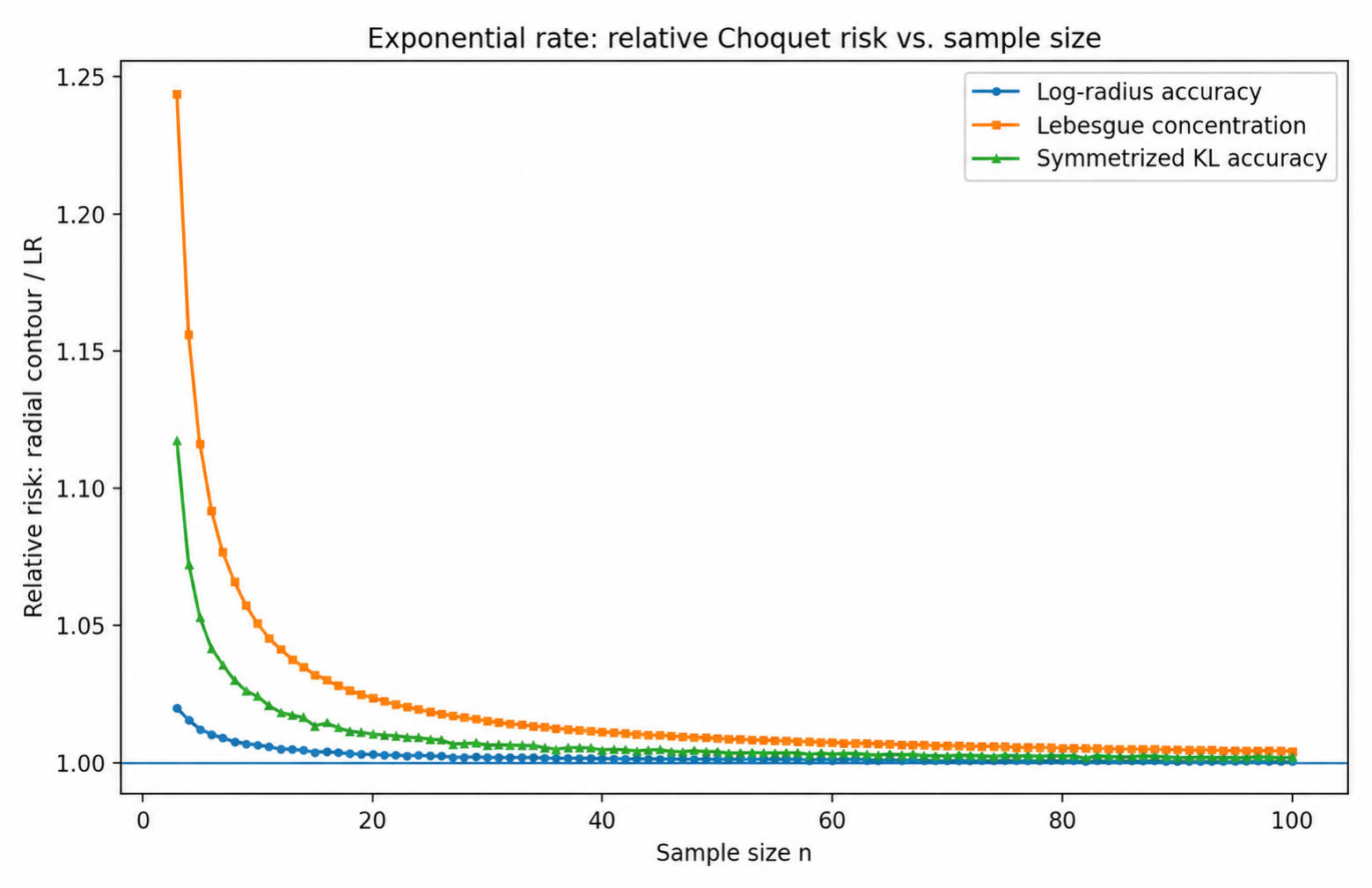}
    \caption{Exponential rate model: relative Choquet risk of the radial contour
from Proposition \ref{prop:validifiedradius} against the calibrated LR contour.  Values above one
indicate larger risk for the radial contour.  The comparison is scalar for each
$n$ because scale equivariance removes dependence on the unknown rate.}
\label{fig:exp-rate-relative-risk}
\end{figure}

\begin{example}[Normal mean with unknown variance: extension vs profile]\label{ex:norm-mean} \citet{Martin2025} revisits Fisher's analysis of Darwin's paired data under a normal model, focusing on inference for the mean $\phi=g(\theta)=\theta_1$ when $\theta=(\theta_1, \theta_2)=(\mu,\sigma^2)$ includes the nuisance scale parameter. In that example, the review compares \emph{extension-based} and \emph{profile-based} marginal IM contours for $\mu$, and notes that the profile-based contour coincides with the two-sided Student-t p-value function and result in a tighter (more efficient) marginal contour.

More generally, the review defines extension-based marginalization by
\begin{equation*}
    \pi_x^{ext}(\varphi)=\sup_{\vartheta:g(\vartheta)=\varphi} \pi_x(\vartheta)
\end{equation*}
    and notes that—while this is simple and preserves validity—one should typically expect it to be \emph{conservative}. 

For profile-based marginalization, define first the profile the relative likelihood
\begin{equation*}
    LR_{prof}(x, \varphi)=\sup_{\vartheta:g(\vartheta)=\varphi} LR(x, \vartheta).
\end{equation*}
Then, followed by validification, we find the profile-marginal contour of the form
\begin{equation*}
    \pi_x^{prof}(\varphi)=\sup_{\vartheta:g(\vartheta)=\varphi}P_\vartheta\bigl\{LR_{prof}(X, \varphi)\leq LR_{prof}(x, \varphi) \bigr\}.
\end{equation*}

The review compares these strategies in the Fisher example for the \emph{fixed} data set of Darwin's, by plotting the different contour functions next to each other. In their case, the comparison is quite clear, as the extension-based contour fully dominates the profiled contour, i.e. is larger everywhere. We complement this with a \emph{procedure-level} comparison: we replace Darwin's fixed dataset by a controlled Monte Carlo experiment as follows.

We assume the data model $X_i\sim N(\mu,\sigma^2)$ iid for $i={1,\dots,n}$ with $n=15$. We further impose the squared error as penalty functional $\Gamma_\theta(\vartheta)=(\vartheta-\theta)^2$ for the decision problem, which is natural for a location parameter. A Monte-Carlo simulation is then employed to compare risks over a grid of $(\mu, \sigma^2)$. By translation equivariance, the risks do not depend on the value of \(\mu\).
Moreover, under the rescaling \(X_i=\mu+\sigma Z_i\), the marginal
\(\alpha\)-cuts for \(\mu\) scale linearly with \(\sigma\). Since the penalty is
squared error, the corresponding Choquet risks scale quadratically:
\[
  R\{(\mu,\sigma^2),\Pi^m\}
  =
  \sigma^2 R\{(0,1),\Pi^m\},
  \qquad
  m\in\{\mathrm{ext},\mathrm{prof}\}.
\]
It is therefore enough to compare the procedures as a function of the scale
parameter \(\sigma\) (equivalently, as a function of the variance \(\sigma^2\)). Figure \ref{fig:choquet-normal}(a) shows that the profile-based marginal
contour has smaller Choquet risk throughout the grid. The increasing vertical
separation between the absolute-risk curves should not be interpreted as a
deterioration of the relative-likelihood ordering or of the calibration as
\(\sigma\) grows. It is simply the scale effect induced by squared-error loss:
larger \(\sigma\) produces wider plausibility regions for \(\mu\), and squared
error measures those widths in squared units. After normalizing by the
profile-based risk, this common \(\sigma^2\) factor cancels. The relative risk is
therefore essentially constant, showing that the efficiency loss from using
extension-based rather than profile-based marginalisation is stable across the
variance grid. In this experiment, extension costs about \(1.54\) times as much
as profiling under the chosen loss.
    
\end{example}

\begin{figure}[H]
  \centering
  \begin{subfigure}[t]{0.49\textwidth}
    \centering
    \includegraphics[width=\linewidth]{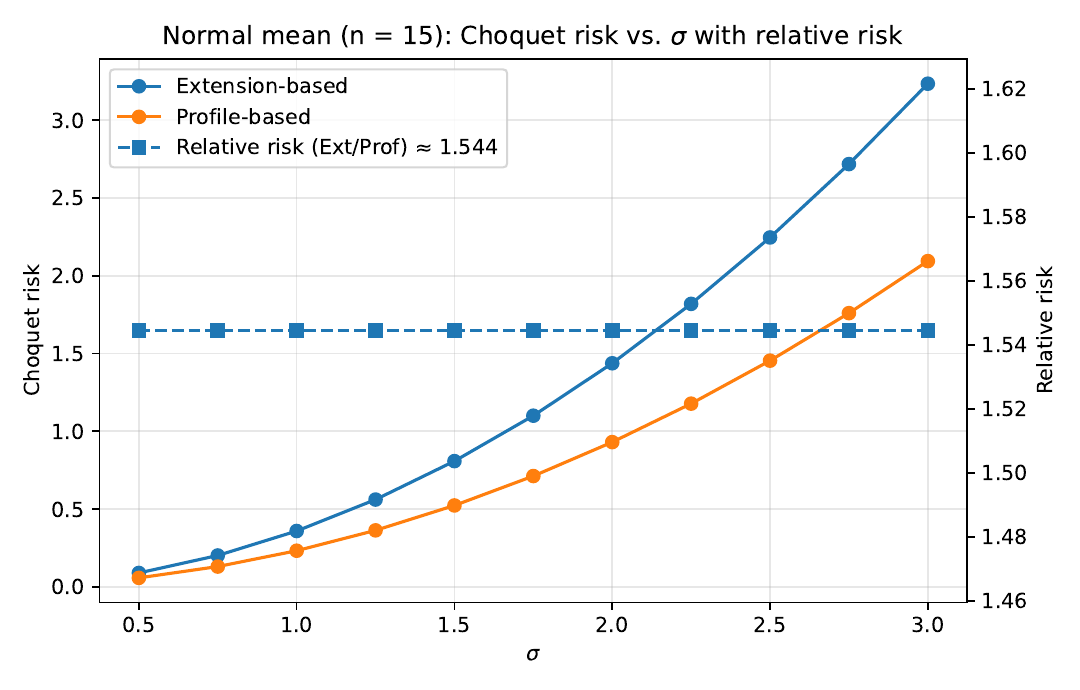}
    \caption{Normal mean ($n=15$): Choquet risk vs.\ $\sigma$.}
    \label{fig:choquet-normal}
  \end{subfigure}\hfill
  \begin{subfigure}[t]{0.49\textwidth}
    \centering
    \includegraphics[width=\linewidth, height=0.64\linewidth]{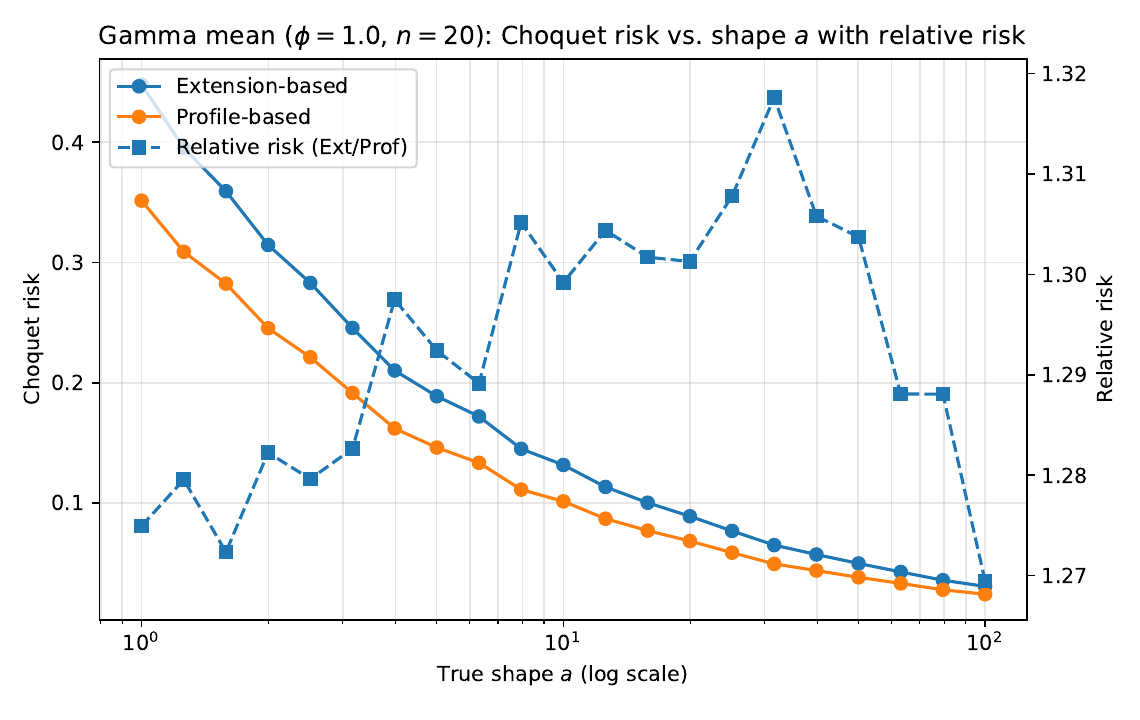}
    \caption{Gamma mean ($\phi=1.0$, $n=20$): Choquet risk vs.\ shape $a$.}
    \label{fig:choquet-gamma}
  \end{subfigure}

  \caption{Comparison of extension-based and profile-based Choquet risk estimates across settings.}
  \label{fig:choquet-side-by-side}
\end{figure}

\begin{example}[Gamma mean: extension vs profile direct comparison.]
    Let $X_1,\dots,X_n$ be i.i.d.\ from a Gamma distribution with unknown shape--scale parameter
$\vartheta=(a,b)\in(0,\infty)^2$. The mean is
\[
\varphi = m(\vartheta) := ab \in \Phi=(0,\infty),
\]
and the remaining degree of freedom is nuisance. In this problem, inference for $\varphi$ is known to be
surprisingly subtle, and likelihood-based marginal IM constructions are conceptually straightforward but
computationally non-trivial because the relevant profile-likelihood sampling distribution depends on the
nuisance parameter. (See Example~6 in the IM review for additional discussion.)%

Consider again the two marginalization strategies Example \ref{ex:norm-mean}, but now applied to the joint relative-likelihood contour for the parameter $\vartheta$ in the Gamma model. In the Gamma mean problem, the sampling distribution
of $R_{\mathrm{prof}}(X,\phi)$ depends on the nuisance parameter, so the calibration step for $\pi_x^{\mathrm{prof}}$ is carried out
numerically (Monte Carlo).

We evaluate the two procedures via Choquet risk under the log-loss penalty
\[
\Gamma_{\varphi}(\phi) := \bigl|\log(\phi/\varphi)\bigr|,\qquad \phi\in\Phi.
\]
This choice is motivated by the fact that the mean is a positive parameter and errors are therefore more naturally assessed on a multiplicative rather than an additive scale. In particular, $\Gamma_\varphi(\phi)$ is invariant under changes of units, since
\[
|\log\{(c\phi)/(c\varphi)\}|=|\log(\phi/\varphi)|,
\]
and it treats over- and under-estimation by the same factor symmetrically:
\[
\Gamma_\varphi(c\varphi)=\Gamma_\varphi(\varphi/c)=\log c.
\]

For $m\in\{\mathrm{ext},\mathrm{prof}\}$, let $A^m_\alpha(x)=\{\phi:\pi_x^m(\phi)\ge \alpha\}$ denote the $\alpha$-cut of the marginal
contour. The Choquet loss and risk are
\[
L(\varphi,\Pi^m_x)
:=\int_0^1 \sup_{\phi\in A^m_\alpha(x)} \Gamma_{\varphi}(\phi)\,d\alpha,
\qquad
R(\vartheta,\Pi^m)
:= E_{\vartheta}\!\left\{L\bigl(m(\vartheta),\Pi^m_X\bigr)\right\}.
\]
Since $A^m_\alpha(x)$ is typically a connected interval in $\Phi$, the inner supremum is attained at an endpoint,
so $L(\varphi,\Pi_x^m)$ can be computed efficiently once the $\alpha$-cuts are available. To compare methods on a
dimensionless scale, we also consider the (pointwise) relative risk
\[
\mathrm{RR}(\vartheta) := \frac{R(\vartheta,\Pi^{\mathrm{ext}})}{R(\vartheta,\Pi^{\mathrm{prof}})}.
\]
Values $\mathrm{RR}(\vartheta)>1$ indicate an efficiency loss from using extension-based marginalisation relative
to profiling. 

\paragraph{Simulation design}
We fix $n=20$ and evaluate risks over a grid of nuisance values by holding the true mean at $\varphi=1$ and varying
the true shape $a$. Specifically, for each $a$ on a log-scale grid we set $b=1/a$ (so that $ab=1$), generate repeated samples
$X^{(k)}\sim \mathrm{Gamma}(a,b)$, compute $\pi^{\mathrm{ext}}_{X^{(k)}}$ and $\pi^{\mathrm{prof}}_{X^{(k)}}$, and approximate the Choquet loss
by numerical integration over a fine grid $\alpha\in[\alpha_{\min},1]$ (with $\alpha_{\min}=10^{-4}$). Averaging over Monte Carlo
replicates yields estimates of $R((a,1/a),\Pi^{\mathrm{ext}})$ and $R((a,1/a),\Pi^{\mathrm{prof}})$, and hence $\mathrm{RR}(a,1/a)$, shown in Figure \ref{fig:choquet-gamma}.

The results in Figure \ref{fig:choquet-gamma} give a clear finite-sample preference for the
profile-based marginal IM. Across the full range of nuisance values considered,
the estimated Choquet risk of the profile-based procedure is below that of the
extension-based procedure. Both risks decrease as the true shape parameter
\(a\) increases. This is expected: throughout the experiment we keep the mean
fixed at \(\phi=ab=1\), so \(b=1/a\), and hence
\[
        \operatorname{Var}_{(a,1/a)}(X_i)=ab^2=\frac{1}{a}.
\]
Larger values of \(a\) therefore correspond to less variable, and less skewed,
Gamma samples, so both marginal contours become tighter under the log-relative
loss.

The absolute separation between the two risks is largest for small values of
\(a\), where the Gamma distribution is most skewed and the nuisance effect is
most pronounced. As \(a\) increases, both procedures improve and the absolute
gap narrows. On the relative-risk scale, however, the efficiency loss from
extension-based marginalisation is fairly stable: throughout the grid,
\(RR(a,1/a)>1\), with values roughly between \(1.27\) and \(1.32\). Thus, in
this experiment, extension-based marginalisation incurs approximately a
\(25\)--\(30\%\) larger Choquet risk than profiling. The small oscillations in
the relative-risk curve should not be over-interpreted, since they are likely
Monte Carlo noise from estimating risks and taking their ratio. The main conclusion is the persistent separation between the two risk
curves, not the exact local shape of the relative-risk curve.
\end{example}

\begin{remark}
In the above comparisons, the results are unambiguous and visually clear: the validified LR contour construction is more efficient, and the profile-based marginal procedures has
smaller estimated Choquet risk throughout the nuisance grid. More generally,
however, when pointwise risk depends on nuisance parameters, \(RR(\vartheta)\)
only gives a pointwise comparison. In other examples, one procedure may be
better for some nuisance values and worse for others. When this is the case, one can aggregate risk over the nuisance, such as by arithmetic or geometric means (depending on the parameter). Such summaries can be useful descriptively, but they necessarily introduce yet another level of subjectivity into the analysis.
\end{remark}

\section{Optimality invariance under reparametrisation}\label{sec:reparinvariance}
Let $\Theta$ be the original parameter space and let $g:\Theta \rightarrow\Phi$, with $\phi=g(\theta)$ a bijective measurable reparametrisation with inverse $\theta=g^{-1}(\phi)$. A possibility measure $\Pi_x$ on $\Theta$ is  meant to summarise ``all available inference'' about $\theta$ given $x$. In that spirit, a mere coordinate change should not change inferential content.

At the level of possibility measure, this is exactly the usual extension / push-forward principle. Define a possibility on $\Phi:=g(\Theta)$ by
\begin{equation*}
    \Pi_x^\phi(B):=\Pi_x(g^{-1}(B)), \qquad  B\subseteq\Phi.
\end{equation*}
Equivalently, at the contour level,
\begin{equation*}
    \pi_x^\phi(\varphi)=\sup_{\vartheta:g(\vartheta)=\varphi}\pi_x(\vartheta)=\pi_x(g^{-1}(\varphi)),
\end{equation*}
and the $\alpha$-cuts transform by
\begin{equation*}
    A_\alpha^\phi(x)=\{\varphi:\pi_x^\phi(\varphi)\geq \alpha\}=g(A_\alpha(x)).
\end{equation*}
So validity and ``structural'' properties expressed via $\alpha$-cuts (nestedness, coverage, unbiasedness, etc.) are preserved under bijections: the inferential object is essentially the same, just expressed in new coordinates.

However, Choquet-risk optimality is not purely an inferential property of $\Pi_x$; it is a property relative to a \emph{penalty} $\Gamma$ used to evaluate a procedure. This is where reparameterization can cause issues unless handled properly.

If we view the reparameterization as a mere change of coordinates, then the \emph{decision problem itself} should be invariant. Concretely, this means: \emph{push forward the penalty along with the parameter}.

Suppose the original penalty is $\Gamma_\theta(\vartheta)$ on $\Theta$. Define the parameterized penalty on $\Phi$ by the natural transport
\begin{equation*}
    \Gamma(x, \phi, \varphi):=\Gamma(x,g^{-1}(\phi), g^{-1}(\varphi)).
\end{equation*}
Then the Choquet loss computed in the new coordinates coincides with the original one. The supremum of the penalty over $g(A_\alpha(x))$ is the same after pulling back by $g^{-1}$, and integrating over $\alpha$ gives equality of Choquet losses, hence equality of risks. 

Under this ``transform-$\Gamma$'' convention, \emph{Choquet risk is invariant} under bijective reparametrisations, and therefore if $\Pi^*$ is Choquet-risk optimal in $(\Theta, \Gamma)$ over a class $\Pi_{\mathcal C}$, then $\Pi^{*\phi}$ is Choquet-risk optimal in $(\Phi, \Gamma_\phi)$ over the push-forward class $\Pi_{\mathcal C}^\phi:=\{\Pi^\phi:\Pi\in \Pi_{\mathcal C}\}$. This holds whenever $g$ is bijective; the class $\Pi_{\mathcal C}$ is stable under push-forward, and the penalty is transformed as above.

If we instead reparametrise $\theta \mapsto \phi$ but keep the same functional form of $\Gamma$ in the new coordinates (e.g. ``squared error in $\phi$'' instead of ``squared error in $\theta$''), then we have \emph{not} performed a coordinate change; we have changed the loss, hence \emph{changed the decision problem}.

This is not a pathology—it is exactly what should happen when the loss encodes domain preference. But it does mean that ``optimality invariance under reparametrisation'' is only meaningful after declaring whether the loss should be \emph{coordinate-dependent} (parameter has physical meaning / preferred scale), or \emph{coordinate-free} (pure inference summary; coordinate system shouldn't matter).

In the IM philosophy, $\Pi_x$ is supposed to summarise \emph{all available inference} without injecting prior information (at least in the vacuous treatment here). A coordinate-dependent penalty can smuggle in preferences about scale/geometry that are not determined by the inference problem—unless the parameter itself comes with a physical or scientific meaning that makes a particular scale ``canonical'' or ``natural''.

So absent such external meaning, it is attractive to choose penalties that are \emph{intrinsic} (coordinate-free) or at least \emph{canonically induced} by the model, so that ``efficiency comparisons'' do not depend on arbitrary parametrisations. 

\subsection{Invariant scale/size penalties}

The most direct way to make a concentration criterion compatible with model invariance is to
replace ordinary Lebesgue size by an invariant notion of size. A key observation is that the
reduction in Corollary \ref{cor:concentration-optimal} uses only monotonicity under inclusion of the $\alpha$-cuts; it
does not depend on any specifically Euclidean notion of length or volume. Thus, in a
transformation model, if $\mu_{\mathrm{Haar}}$ denotes a Haar measure, or more generally any
measure induced by the group action that is invariant under the relevant transformations, we
may define the invariant concentration penalty
\begin{equation}
\Gamma_x^{\mathrm{inv}}(\vartheta)
=
\phi\!\left[
\mu_{\mathrm{Haar}}\!\bigl\{A_{\pi_x(\vartheta)}(x)\bigr\}
\right].
\end{equation}
Exactly as in Section \ref{sec:choclossandrisk}, nestedness of the $\alpha$-cuts implies that the corresponding
Choquet loss reduces to
\[
L^{\mathrm{inv}}(\theta,\Pi_x)
=
\int_0^1
\phi\!\left[
\mu_{\mathrm{Haar}}\!\bigl\{A_{\pi_x(\vartheta)}(x)\bigr\}
\right]\,ds,
\]
and therefore
\[
R^{\mathrm{inv}}(\theta,\Pi)
=
\int_0^1
E_\theta\!\left\{
\phi\!\left[
\mu_{\mathrm{Haar}}\!\bigl\{A_{\pi_x(\vartheta)}(x)\bigr\}
\right]
\right\}ds.
\]

In the special case $\phi(\ell)=\ell$, the Choquet risk reduces to integrated
expected Haar-size of the $\alpha$-cuts. Thus every levelwise size-optimality result for
confidence sets carries over verbatim after replacing Lebesgue size by invariant size. In
particular, whenever one can identify, for each $\alpha$, a $G$-equivariant
$(1-\alpha)$-confidence set of smallest expected Haar size within the relevant class,
stacking those sets yields a Choquet-risk optimal valid equivariant contour for the invariant
concentration criterion.

This is the natural notion of concentration in transformation models. Ordinary Lebesgue
length or volume depends on the chosen coordinates and need not respect the symmetry of the
problem, whereas Haar size measures concentration in units intrinsic to the group action.

In some especially tractable invariant problems, one can go further. The levelwise optimal
equivariant confidence sets are then superlevel sets of a single reduced-space score, so the
optimal contour can be obtained directly by validifying that score, rather than by first
stacking the optimal sets one level at a time.

\begin{example}[Variance ratio: a direct invariant marginal contour]
Let $X_1,\ldots,X_m \sim N(\xi,\sigma^2)$ and $Y_1,\ldots,Y_n \sim N(\eta,\tau^2)$
independently, denote by $S_X^2$ and $S_Y^2$ the respective sample variances, and consider inference for the invariant ratio
\[
\Delta := \tau^2/\sigma^2.
\]
By sufficiency and invariance under separate translations and scales, \citet{lehmann2005testing}[Example 6.12.4] show that the equivariant confidence problem for $\Delta$ reduces
to a one-dimensional pivotal reduction $W_\Delta(X)$; equivalently, one may take
$W_\Delta(X)=\Delta/V$, up to inversion, with $V=S_Y^2/S_X^2$. Let $p$ denote the
common density of this reduced statistic, say under $\Delta=1$.

If size on the reduced space is measured by Lebesgue measure, the average-smallest
equivariant $(1-\alpha)$-confidence sets are obtained by thresholding $p$. If instead size is
measured by the invariant/Haar criterion, the relevant reduced-space density is
$\rho(w)=1/w$, so the corresponding invariant score is
\[
q_H(w)
=
\frac{p(w)}{\rho(w)}
=
w\,p(w).
\]
Thus the Haar-smallest equivariant confidence sets are superlevel sets of $q_H$, not of the
ordinary reduced-space density.

Let $F_H$ denote the distribution function of $q_H\{W_\Delta(X)\}$ under $P_\Delta$.
Since $W_\Delta(X)$ is pivotal, this law does not depend on $\Delta$. Assuming continuity,
let $c_\alpha$ be defined by
\[
P_\Delta\!\left\{
q_H\!\left(W_\Delta(X)\right) \ge c_\alpha
\right\}
=
1-\alpha.
\]
Then the average-smallest equivariant $(1-\alpha)$-confidence set is
\[
C_\alpha^\star(x)
=
\{\Delta : q_H(W_\Delta(x)) \ge c_\alpha\},
\]
and the family $\{C_\alpha^\star(x)\}_{\alpha\in(0,1)}$ is nested in $\alpha$.

This suggests defining the contour directly from the reduced score by
\[
r_x^H(\Delta) := q_H\!\left(W_\Delta(x)\right),
\qquad
\pi_x^H(\Delta)
:=
P_\Delta\!\left\{
r_X^H(\Delta) \le r_x^H(\Delta)
\right\}.
\]
Because $\Delta \mapsto W_\Delta(x)$ ranges over the reduced space, we have
\[
\sup_\Delta r_x^H(\Delta)
=
\sup_w q_H(w),
\]
which is finite and does not depend on $x$. Hence $\pi^H$ is a genuine possibility contour.
Moreover, by the probability integral transform it is valid, and since validification is
monotone in the underlying score,
\[
\{\Delta : \pi_x^H(\Delta) \ge \alpha\}
=
\{\Delta : q_H(W_\Delta(x)) \ge c_\alpha\}
=
C_\alpha^\star(x).
\]
So the $\alpha$-cuts of the contour are exactly the Haar-smallest equivariant confidence
sets.

By Lemma \ref{lem: equiv**2}, $\pi^H$ is therefore a valid $G$-equivariant possibility contour. By
Corollary \ref{cor:concentration-optimal} (with Lebesgue size exhanged for Haar size), the induced possibility measure is pointwise Choquet-risk optimal for the
invariant concentration penalty above, among valid $G$-equivariant contours with finite
Haar-sized $\alpha$-cuts. In particular, $\pi^H$ is Haar-concentration-optimal.

\end{example}

Outside strict group-invariant settings, the same idea still points toward a natural
replacement for Lebesgue size: one may measure set size using an intrinsic volume element
rather than arbitrary coordinate volume. In regular parametric models, the canonical choice
is the Jeffreys/Fisher volume element, i.e.\ the Riemannian volume induced by Fisher
information. Its key feature is coordinate invariance: under reparametrisation, the Jacobian
factors are exactly offset by the transformation of the Fisher information. This leads
naturally to the broader class of intrinsic penalties considered next.

\subsection{Intrinsic divergence penalties} 

A second route to reparametrisation-robust efficiency is to define the accuracy loss directly on the statistical model, rather than on a chosen coordinate system. Let \(D(P_\theta,P_\vartheta)\) be a nonnegative divergence on the sampling laws, and let \(\psi:[0,\infty)\to[0,\infty)\) be continuous and nondecreasing, with \(\psi(0)=0\). We then consider penalties of the form
\[
\Gamma_\theta(\vartheta):=\psi\!\left(D(P_\theta,P_\vartheta)\right).
\] Given such a penalty, the associated ``intrinsic Choquet loss'' is obtained by the upper expectation as before: \[L(\theta, \Pi_x)=\int_0^1\sup_{\vartheta \in A_\alpha(x)}\psi(D(P_\theta, P_\vartheta))d\alpha.\]

Because \(\Gamma_\theta(\vartheta)\) depends only on the pair of sampling distributions \((P_\theta,P_\vartheta)\), it is invariant under any bijective reparametrisation of the parameter. A change of coordinates relabels \(\theta\), but does not alter the underlying family \(\{P_\theta\}\). In this sense, the penalty is intrinsic to the statistical problem rather than to an arbitrary parametrisation.

This choice is also natural from the standpoint of likelihood-based IMs. Likelihood ratios rank candidate parameter values according to how distinguishable the corresponding sampling laws are from the truth. A divergence-based loss evaluates the same phenomenon on the decision side: it penalises distributional discrepancy between \(P_\theta\) and \(P_\vartheta\), rather than coordinate discrepancy between \(\theta\) and \(\vartheta\).

Assume now that \(D\) is smooth enough that, as \(\delta\to 0\),
\[
D(P_{\theta+\delta},P_\theta)
=
c\,\delta^\top I(\theta)\delta + o(\|\delta\|^2),
\]
for some constant \(c>0\), where \(I(\theta)\) is the Fisher information. Thus \(D\) induces, to second order, the Fisher--Rao metric. No global symmetry of \(D\) is needed here; only its local quadratic part matters.

Fix \(\theta_0\), and introduce the usual Fisher-normalised local parameter
\[
\theta_{n,u}
=
\theta_0 + n^{-1/2} I(\theta_0)^{-1/2}u,
\qquad u\in\mathbb{R}^d.
\]
The corresponding local penalty is
\[
\ell_n(u,v)
:=
\psi\!\left(n\,D\!\left(P_{\theta_{n,u}},P_{\theta_{n,v}}\right)\right).
\]
Whenever the local expansion above is uniform on compact subsets of \(\mathbb{R}^d\times\mathbb{R}^d\), we obtain
\[
\ell_n(u,v)\to h_\infty(\|u-v\|)
\qquad\text{uniformly on compacts,}
\]
where
\[
h_\infty(r):=\psi(c r^2).
\]
Thus, after Fisher normalization, a divergence-based intrinsic penalty becomes asymptotically radial in the local parameter.

This observation identifies the relevant limit decision problem: once the loss is localized, the Gaussian shift experiment becomes the natural benchmark, and the exact Gaussian-shift Choquet-risk result from the equivariance theory is the canonical target.

\begin{remark}
A practical issue with divergence-based penalties is the integrability of the
resulting Choquet loss. Kullback--Leibler divergence is locally natural, since
its second-order expansion recovers the Fisher--Rao quadratic form, but it can
behave poorly globally: in low-plausibility regions, or near the boundary of
the model, the supremum over an $\alpha$-cut may be infinite even when the local
geometry is well behaved. To avoid this phenomenon a particularly natural choice is squared
Hellinger distance,
\[
h^2(P_\theta,P_\vartheta)
=
1-\int \sqrt{dP_\theta\,dP_\vartheta},
\]
which lies in $[0,1]$ and therefore guarantees existence of the Choquet loss
for every valid contour. It is also conceptually well aligned with the present
goal: the penalty measures how distinguishable the sampling law $P_\vartheta$
is from the true law $P_\theta$, rather than how far apart two parameter labels
are in an arbitrary coordinate system. Moreover, in regular dominated models,
squared Hellinger distance has the same local Fisher--Rao quadratic expansion
as Kullback--Leibler divergence, up to a constant, while avoiding the latter's
global explosion.
\end{remark}

\subsection{Asymptotic outlook under Fisher--Rao local loss}

The Gaussian-shift result from Section \ref{sec:Gaussball} gives an exact solution in the normal translation experiment. The possibilistic Bernstein--von Mises theorem further identifies this contour as the asymptotic limit of the validified LR contour. Ideally, we would want to show that Choquet-risk statements converge similarly to this natural benchmark in the limit experiment. We briefly sketch the outlook for showing such a result.

Suppose that a sequence of valid IM contours satisfies a possibilistic Bernstein--von Mises theorem after centering and Fisher normalization, in the sense discussed by \citet{Martin2025BvM}. Write \(\check\pi_n\) for the resulting localized contour, suppressing its data dependence in the notation, and let
\[
A^{(n)}_\alpha:=\{v:\check\pi_n(v)\ge \alpha\}
\]
denote its \(\alpha\)-cut. Combined with the local divergence expansion from the previous subsection, this suggests the following asymptotic picture: with
\[
\ell_n(u,v)
=
\psi\!\left(n\,D\!\left(P_{\theta_{n,u}},P_{\theta_{n,v}}\right)\right),
\]
one expects the local Choquet risk of the finite-sample IM to converge to the exact Gaussian benchmark. More precisely, with
\[
R_n^{\mathrm{loc}}(u,\check\pi_n)
:=
E_{\theta_{n,u}}
\left[
\int_0^1
\sup_{v\in A^{(n)}_\alpha}
\ell_n(u,v)\,d\alpha
\right]
\]
and
\[
R_\infty(u,\pi^\star)
:=
E_u
\left[
\int_0^1
\sup_{v\in A^\star_\alpha(Z)}
h_\infty(\|u-v\|)\,d\alpha
\right],
\]
the natural expected limit statement is
\[
\sup_{\|u\|\le M}
\left|
R_n^{\mathrm{loc}}(u,\check\pi_n)-R_\infty(u,\pi^\star)
\right|
\to 0
\qquad\text{for each }M<\infty.
\]
In words, the localized Choquet risk should converge to the exact Gaussian-shift risk singled out by the equivariant theory.

The possibilistic Bernstein--von Mises theorem gives contour-level Gaussian approximation, locally on compact sets and in probability. To upgrade this statement to convergence of Choquet risks, one must additionally prove convergence of the induced $\alpha$-cuts away from problematic levels, control the tails coming both from the random limiting center and from the low-$\alpha$ regime, and establish sufficient uniform integrability to interchange the contour limit with the $\alpha$-integral and the sampling expectation. After this upper-bound step, a full local asymptotic minimax theorem would still require a matching lower bound: via LAN/Le Cam comparison, every finite-sample valid contour sequence should induce an asymptotically valid contour in the Gaussian shift experiment, so that the exact Gaussian minimax bound applies. Thus, at present, the Fisher--Rao/divergence connection should be read as an asymptotic roadmap rather than as a completed minimax theorem.

\section{Conclusion}\label{sec:conclusion}

This paper has proposed Choquet risk as a finite-sample efficiency criterion for valid possibilistic inferential models. The starting point was that validity alone is not enough to choose among calibrated possibility contours: distinct valid IMs can differ substantially in informativeness, while pointwise dominance is too strong to provide a general ordering. By evaluating the data-dependent possibility measure itself through Choquet, or upper, expectation, the proposed framework gives a procedure-level risk criterion that respects the imprecision responsible for exact finite-sample calibration.

The main technical simplification is the $\alpha$-cut representation. Choquet loss decomposes into an integral, over plausibility levels, of the worst penalty over the corresponding calibrated confidence set; Choquet risk is the sampling expectation of this loss. This representation makes the framework both interpretable and operational. For concentration penalties, it reduces to an integrated expected-size criterion, so classical confidence-set optimality results can be lifted directly to possibility contours. More generally, the credal-set interpretation shows that Choquet loss is a least-favorable confidence loss, and identifies the sense in which the present construction extends confidence risk from additive confidence distributions to non-additive, calibrated IM output.

The resulting theory gives two complementary routes to finite-sample optimality. The first is through unbiasedness: under validity, the possibilistic unbiasedness condition introduced here is equivalent to the usual unbiasedness of the induced confidence sets and tests, so UMPU and most-accurate-unbiased theory can be translated into concentration-optimal valid contours. The second is through equivariance: equivariant contours are equivalent to equivariant $\alpha$-cuts, a Hunt--Stein-type reduction justifies minimax comparisons within an equivariant subclass, and the Gaussian location experiment yields an exact radial Choquet-risk minimax result for the standard Gaussian possibility contour. These results show that Choquet risk does more than reproduce familiar interval-length comparisons; it supplies a decision-theoretic language in which classical optimality, imprecise uncertainty, and finite-sample validity can be treated in a common framework.

The framework also clarifies what kind of efficiency statements should not be expected. Choquet risk is risk relative to a penalty, so there is no coordinate-free or universally best contour without first specifying what kind of inferential error or concentration is being penalised. Reparameterization changes the decision problem unless the penalty is transported along with the parameterization. This motivates invariant size criteria and intrinsic divergence losses, which avoid tying the comparison to arbitrary coordinates. In regular models, divergence losses connect locally to Fisher--Rao quadratic loss and hence to the Gaussian benchmark suggested by the possibilistic Bernstein--von Mises theorem. In the present paper this connection is deliberately used only as an asymptotic guide. A full local asymptotic minimax theorem for valid possibilistic IMs under intrinsic Choquet-risk losses remains open. More precisely, one would like to show that, in regular experiments and under suitable localisation, the finite-sample Choquet-risk comparison converges to the Gaussian limit experiment in a way that transfers the Gaussian minimax result to a local asymptotic minimax statement. Establishing such a theorem would turn the roadmap developed here into a complete asymptotic efficiency theory, and is a natural target for future work.

Finally, Choquet risk provides a practical basis for direct comparisons when analytic optimality arguments are unavailable. The simulation comparisons illustrate how procedure-level risk can quantify differences that are otherwise expressed only through selected contour plots or fixed-level confidence-region sizes. In particular, the extension-versus-profile marginalization examples show how competing valid possibilistic procedures can be compared on a common loss scale, and how relative risk can make the efficiency cost of one construction visible. These examples also point to a practical limitation: estimating Choquet risk can be computationally demanding, especially when $\alpha$-cut optimization must be repeated over nuisance-dependent risk surfaces.

Taken together, the results support the central message that valid possibility measures can be assessed by the same kind of decision-theoretic discipline used to compare ordinary statistical procedures, provided the loss is adapted to the non-additive nature of the inferential output. Choquet risk supplies such an adaptation. It preserves the nested confidence-set structure built into valid IMs, connects finite-sample comparisons to classical optimality theory, and provides a bridge from exact calibration to the asymptotic Gaussian efficiency picture. The remaining challenges are to sharpen this bridge through local asymptotic minimax theory, to develop more efficient risk-estimation methods, and to investigate intrinsic penalties for broader model classes.

\section*{Acknowledgements}
I am grateful to Ryan Martin for helpful discussions and feedback on earlier drafts. I also thank Rolf Larsson for comments and suggestions that improved the presentation. Any remaining errors are my own.

\section*{Declaration of generative AI use}
During the preparation of this manuscript, the author has used ChatGPT for language editing, LaTeX troubleshooting, feedback on exposition and organization, literature search, and ideas testing and code implementation for examples and simulation. The author notes in particular that the proof device via Minkowski symmetrization in \ref{prop:gaussian-shift-balls}--originally from PDE-theory(for more, see e.g. \citet{Daners2011Krahn})--was suggested by ChatGPT. The author reviewed and edited all AI-assisted output, and takes full responsibility for the content of the manuscript.

\bibliographystyle{plainnat}
\bibliography{sample}

@book{schweder_hjort_2016, place={Cambridge}, series={Cambridge Series in Statistical and Probabilistic Mathematics}, title={Confidence, Likelihood, Probability: Statistical Inference with Confidence Distributions}, DOI={10.1017/CBO9781139046671}, publisher={Cambridge University Press}, author={Schweder, Tore and Hjort, Nils Lid}, year={2016}, collection={Cambridge Series in Statistical and Probabilistic Mathematics}}

@article{BalchMartinFerson2019FCT,
  author    = {Balch, Michael S. and Martin, Ryan and Ferson, Scott},
  title     = {Satellite conjunction analysis and the false confidence theorem},
  journal   = {Proceedings of the Royal Society A: Mathematical, Physical and Engineering Sciences},
  year      = {2019},
  volume    = {475},
  number    = {2227},
  pages     = {20180565},
  doi       = {10.1098/rspa.2018.0565},
  url       = {https://doi.org/10.1098/rspa.2018.0565}
}

@book{lehmann2005testing,
  title     = {Testing Statistical Hypotheses},
  author    = {Lehmann, Erich L. and Romano, Joseph P.},
  edition   = {3rd},
  year      = {2005},
  publisher = {Springer},
  address   = {New York},
  isbn      = {978-0387269373}
}

@book{Eaton1989,
  author    = {Eaton, Morris L.},
  title     = {Group Invariance Applications in Statistics},
  year      = {1989},
  publisher = {Institute of Mathematical Statistics},
  series    = {NSF-CBMS Regional Conference Series in Probability and Statistics},
  volume    = {1},
  address   = {Hayward, CA}
}

@article{efron2010future,
  title={The future of indirect evidence},
  author={Efron, Bradley},
  journal={Statistical science: a review journal of the Institute of Mathematical Statistics},
  volume={25},
  number={2},
  pages={145},
  year={2010}
}

@article{couso2001necessity,
  title={The necessity of the strong $\alpha$-cuts of a fuzzy set},
  author={Couso, In{\'e}s and Montes, Susana and Gil, Pedro},
  journal={International Journal of Uncertainty, Fuzziness and Knowledge-Based Systems},
  volume={9},
  number={02},
  pages={249--262},
  year={2001},
  publisher={World Scientific}
}

@article{demyst,
author = {Yifan Cui and Jan Hannig},
title = {{Demystifying Inferential Models and Confidence Curves: A Fiducial Perspective}},
volume = {40},
journal = {Statistical Science},
number = {2},
publisher = {Institute of Mathematical Statistics},
pages = {211 -- 218},
year = {2025},
doi = {10.1214/24-STS924},
URL = {https://doi.org/10.1214/24-STS924}
}

@book{dubois_prade_1988,
  author    = {Dubois, Didier and Prade, Henri},
  title     = {Possibility Theory: An Approach to Computerized Processing of Uncertainty},
  publisher = {Plenum Press},
  address   = {New York},
  year      = {1988},
  isbn      = {978-0306425202}
}

@article{dubois2004probability,
  title={Probability-possibility transformations, triangular fuzzy sets, and probabilistic inequalities},
  author={Dubois, Didier and Foulloy, Laurent and Mauris, Gilles and Prade, Henri},
  journal={Reliable computing},
  volume={10},
  number={4},
  pages={273--297},
  year={2004},
  publisher={Springer}
}

@article{Dubois2024likelihood,
title = {Reasoning and learning in the setting of possibility theory - Overview and perspectives},
journal = {International Journal of Approximate Reasoning},
volume = {171},
pages = {109028},
year = {2024},
note = {Synergies between Machine Learning and Reasoning},
issn = {0888-613X},
doi = {https://doi.org/10.1016/j.ijar.2023.109028},
url = {https://www.sciencedirect.com/science/article/pii/S0888613X23001597},
author = {Didier Dubois and Henri Prade}
}

@book{BergerObjective,
author = {Berger, James O and Bernardo, Jose M and Sun, Dongchu},
title = {Objective Bayesian Inference},
publisher = {WORLD SCIENTIFIC},
year = {2024},
doi = {10.1142/13640},
address = {},
edition   = {},
URL = {https://www.worldscientific.com/doi/abs/10.1142/13640},
eprint = {https://www.worldscientific.com/doi/pdf/10.1142/13640}
}

@article{Erhard-orig,
 ISSN = {00255521, 19031807},
 URL = {http://www.jstor.org/stable/24491511},
 author = {ANTOINE EHRHARD},
 journal = {Mathematica Scandinavica},
 number = {2},
 pages = {281--301},
 publisher = {Mathematica Scandinavica},
 title = {SYMÉTRISATION DANS L'ESPACE DE GAUSS},
 urldate = {2026-04-20},
 volume = {53},
 year = {1983}
}

@article{BORELL2003663,
title = {The Ehrhard inequality},
journal = {Comptes Rendus Mathematique},
volume = {337},
number = {10},
pages = {663-666},
year = {2003},
issn = {1631-073X},
doi = {https://doi.org/10.1016/j.crma.2003.09.031},
url = {https://www.sciencedirect.com/science/article/pii/S1631073X03004461},
author = {Christer Borell}
}

@article{Jeffreys46,
 ISSN = {00804630},
 URL = {http://www.jstor.org/stable/97883},
 author = {Harold Jeffreys},
 journal = {Proceedings of the Royal Society of London. Series A, Mathematical and Physical Sciences},
 number = {1007},
 pages = {453--461},
 publisher = {The Royal Society},
 title = {An Invariant Form for the Prior Probability in Estimation Problems},
 urldate = {2026-06-29},
 volume = {186},
 year = {1946}
}

@article{Martin2013,
   title={Inferential Models: A Framework for Prior-Free Posterior Probabilistic Inference},
   volume={108},
   ISSN={1537-274X},
   url={http://dx.doi.org/10.1080/01621459.2012.747960},
   DOI={10.1080/01621459.2012.747960},
   number={501},
   journal={Journal of the American Statistical Association},
   publisher={Informa UK Limited},
   author={Martin, Ryan and Liu, Chuanhai},
   year={2013},
   month=mar, pages={301–313} }

@misc{Martin2023b,
      title={Valid and efficient imprecise-probabilistic inference with partial priors, II. General framework}, 
      author={Ryan Martin},
      year={2023},
      eprint={2211.14567},
      archivePrefix={arXiv},
      primaryClass={stat.ME},
      url={https://arxiv.org/abs/2211.14567}, 
}

@article{Martin2025,
author = {Ryan Martin},
title = {Possibilistic Inferential Models: A Review},
journal = {Journal of the American Statistical Association},
volume = {121},
number = {553},
pages = {807--826},
year = {2026},
publisher = {Taylor \& Francis},
doi = {10.1080/01621459.2025.2606127},


URL = { 
    
        https://doi.org/10.1080/01621459.2025.2606127
    
    

},
eprint = { 
    
        https://doi.org/10.1080/01621459.2025.2606127
    
    

}

}

@article{MartinIMMC,
author = {Ryan Martin},
title = {An efficient Monte Carlo method for valid prior-free possibilistic statistical inference},
journal = {Journal of the American Statistical Association},
volume = {0},
number = {ja},
pages = {1--25},
year = {2026},
publisher = {Taylor \& Francis},
doi = {10.1080/01621459.2026.2671450},


URL = { 
    
        https://doi.org/10.1080/01621459.2026.2671450
    
    

},
eprint = { 
    
        https://doi.org/10.1080/01621459.2026.2671450
    
    

}

}

@misc{reIM,
      title={No-prior Bayesian inference reIMagined: probabilistic approximations of inferential models}, 
      author={Ryan Martin},
      year={2025},
      eprint={2503.19748},
      archivePrefix={arXiv},
      primaryClass={stat.ME},
      url={https://arxiv.org/abs/2503.19748}, 
}

@article{Martin2025BvM,
title = {Asymptotic efficiency of inferential models and a possibilistic Bernstein–von Mises theorem},
journal = {International Journal of Approximate Reasoning},
volume = {180},
pages = {109389},
year = {2025},
issn = {0888-613X},
doi = {https://doi.org/10.1016/j.ijar.2025.109389},
url = {https://www.sciencedirect.com/science/article/pii/S0888613X25000301},
author = {Ryan Martin and Jonathan P. Williams}
}

@misc{Martin2026decision,
      title={Decision-making with possibilistic inferential models}, 
      author={Ryan Martin and Shih-Ni Prim and Jonathan Williams},
      year={2026},
      eprint={2112.13247},
      archivePrefix={arXiv},
      primaryClass={math.ST},
      url={https://arxiv.org/abs/2112.13247}, 
}

@article{MinkowskiSymm,
author = {D. Coupier and Yu. Davydov},
title = {{Random symmetrizations of convex bodies}},
volume = {46},
journal = {Advances in Applied Probability},
number = {3},
publisher = {Applied Probability Trust},
pages = {603 -- 621},
year = {2014},
doi = {10.1239/aap/1409319551},
URL = {https://doi.org/10.1239/aap/1409319551}
}

@book{TroffaesDeCooman2014,
  author    = {Troffaes, Matthias C. M. and de Cooman, Gert},
  title     = {Lower Previsions},
  year      = {2014},
  publisher = {John Wiley \& Sons},
  series    = {Wiley Series in Probability and Statistics},
  isbn      = {978-0-470-72377-7},
  doi       = {10.1002/9781118762622},
  url       = {https://onlinelibrary.wiley.com/doi/book/10.1002/9781118762622}
}

@article{Zadeh1978Possibility,
  author  = {Zadeh, Lotfi A.},
  title   = {Fuzzy sets as a basis for a theory of possibility},
  journal = {Fuzzy Sets and Systems},
  year    = {1978},
  volume  = {1},
  number  = {1},
  pages   = {3--28},
  doi     = {10.1016/0165-0114(78)90029-5}
}

@article{DUBOIS1990consapprox,
title = {Consonant approximations of belief functions},
journal = {International Journal of Approximate Reasoning},
volume = {4},
number = {5},
pages = {419-449},
year = {1990},
issn = {0888-613X},
doi = {https://doi.org/10.1016/0888-613X(90)90015-T},
url = {https://www.sciencedirect.com/science/article/pii/0888613X9090015T},
author = {Didier Dubois and Henri Prade}
}

@article{Dubois2006PossibilityStats,
  author  = {Dubois, Didier},
  title   = {Possibility theory and statistical reasoning},
  journal = {Computational Statistics \& Data Analysis},
  year    = {2006},
  volume  = {51},
  number  = {1},
  pages   = {47--69},
  doi     = {10.1016/j.csda.2006.04.015}
}

@unpublished{Schweder2018Unbiased,
  author = {Schweder, Tore},
  title = {Unbiased Confidence},
  note = {Presentation at the FocuStat Conference 2018, University of Oslo},
  year = {2018},
  url = {https://www.mn.uio.no/math/english/research/projects/focustat/workshops%20and%20conference/FocuStat%20Conference%202018%3A%20V%C3%A5rens%20Vakreste%20Variabler/b1_schweder_unbiasedconfidence.pdf}
}

@article{Pratt01091961,
author = {John W. Pratt},
title = {Length of Confidence Intervals},
journal = {Journal of the American Statistical Association},
volume = {56},
number = {295},
pages = {549--567},
year = {1961},
publisher = {Taylor \& Francis},
doi = {10.1080/01621459.1961.10480644},


URL = { 
    
    
        https://www.tandfonline.com/doi/abs/10.1080/01621459.1961.10480644
    

},
eprint = { 
    
    
        https://www.tandfonline.com/doi/pdf/10.1080/01621459.1961.10480644
    

}

}

@book{shafer2020mathematical,
 ISBN = {9780691100425},
 URL = {http://www.jstor.org/stable/j.ctv10vm1qb},
 author = {Glenn Shafer},
 publisher = {Princeton University Press},
 title = {A Mathematical Theory of Evidence},
 urldate = {2026-05-19},
 year = {1976}
}

@article{Daners2011Krahn,
  author  = {Daners, Daniel},
  title   = {Krahn's proof of the Rayleigh conjecture revisited},
  journal = {Archiv der Mathematik},
  volume  = {96},
  number  = {2},
  year    = {2011},
  pages   = {187--199},
  doi     = {10.1007/s00013-010-0218-x}
}

@article{Hannig2016,
author = {Jan Hannig and Hari Iyer and Randy C. S. Lai and Thomas C. M. Lee and},
title = {Generalized Fiducial Inference: A Review and New Results},
journal = {Journal of the American Statistical Association},
volume = {111},
number = {515},
pages = {1346--1361},
year = {2016},
publisher = {ASA Website},
doi = {10.1080/01621459.2016.1165102},
URL = {https://doi.org/10.1080/01621459.2016.1165102},
eprint = {https://doi.org/10.1080/01621459.2016.1165102}
}

@article{xie2013,
author = {Xie, Min-ge and Singh, Kesar},
title = {Confidence Distribution, the Frequentist Distribution Estimator of a Parameter: A Review},
journal = {International Statistical Review},
volume = {81},
number = {1},
pages = {3-39},
doi = {https://doi.org/10.1111/insr.12000},
url = {https://onlinelibrary.wiley.com/doi/abs/10.1111/insr.12000},
eprint = {https://onlinelibrary.wiley.com/doi/pdf/10.1111/insr.12000},
year = {2013}
}

\appendix
\section{Appendix}
\begin{lemma}[Hunt-Stein symmetrization for valid contours, compact case.]\label{lem:A1}

Let a group $G$ act measurably on the sample space $\mathcal X$ and parameter
space $\Theta$. Assume the model $\{P_\theta:\theta\in\Theta\}$ is $G$-invariant,
and the penalty is $G$-invariant in the sense that
\[
\Gamma_\theta(\vartheta)=h\bigl(\rho(\vartheta,\theta)\bigr),\qquad
\rho(g\cdot\vartheta,g\cdot\theta)=\rho(\vartheta,\theta)
\]
for some nondecreasing $h$ and $\rho: \Theta \times \Theta \rightarrow[0, \infty)$ a measurable, $G$-invariant distance-like function.
Let $\pi$ be a valid possibility contour with $\alpha$-cuts
$A_\alpha(x)=\{\vartheta:\pi_x(\vartheta)\ge\alpha\}$.

For each probability measure $\nu$ on $G$, define the \emph{randomized
symmetrized} contour by
\[
\pi^{(\nu)}_X(\vartheta):=\pi_{G\cdot X}(G\cdot\vartheta),
\qquad G\sim\nu\ \text{independent of }X.
\]
Let $r_\alpha(\theta,\pi)=E_\theta\bigl[\sup_{\vartheta\in A_\alpha(X)}
\Gamma_\theta(\vartheta)\bigr]$ and
$R(\theta,\pi)=\int_0^1 r_\alpha(\theta,\pi)\,d\alpha$ as in § \ref{sec:optRisk}.
Then for any $\nu$:

\begin{enumerate}
\item[(i)] $\pi^{(\nu)}$ is valid: for all $\theta,\alpha$,
\[
P_\theta\{\pi^{(\nu)}_X(\theta)\le\alpha\}\le\alpha.
\]
\item[(ii)] If $G$ is compact and $\nu$ is its normalized left Haar measure
$\mu$, then $\pi^{(\mu)}$ is $G$-equivariant in distribution:
\[
\pi^{(\mu)}_{g\cdot X}(g\cdot\theta)\stackrel{d}{=}\pi^{(\mu)}_X(\theta)
\quad\text{for all }g\in G,\theta\in\Theta.
\]
\item[(iii)] For every $\alpha$ and $\theta$ we have the orbit–average identity
\[
r_\alpha(\theta,\pi^{(\mu)})=\int_G r_\alpha(g\cdot\theta,\pi)\,\mu(dg),
\]
so in particular
\[
\sup_\theta r_\alpha(\theta,\pi^{(\mu)})\le
\sup_\theta r_\alpha(\theta,\pi)
\quad\text{and}\quad
\sup_\theta R(\theta,\pi^{(\mu)})\le\sup_\theta R(\theta,\pi).
\]
\end{enumerate}

\end{lemma}
\begin{proof}
 (i) \emph{Validity.} For every fixed $\theta, \alpha$,
 \begin{equation*}
     P_\theta\left\{\pi_X^{(\nu)}(\theta)\leq \alpha \right\}=\int_GP_\theta\left\{\pi_{gX}^{(\nu)}(g\cdot\theta)\leq \alpha \right\}\nu(dg)=\int_G P_{g\theta}\left\{\pi_X^{(\nu)}(g\cdot\theta)\leq \alpha \right\} \leq \alpha,
 \end{equation*}
 hence the randomized symmetrized contour $\pi_X^{(\nu)}$ is \emph{strongly valid}.

 (ii) \emph{Equivariance in distribution (compact case).} With $\nu=\mu$ the Haar measure, for any $h \in G$, 
 \begin{equation*}
     \pi_{hX}^{(\mu)}(h\cdot\theta)= \pi_{G\cdot(hX)}\left(G\cdot(h\cdot\theta)\right)= \pi_{(G\cdot h)X}\left((G\cdot h)\cdot\theta\right).
 \end{equation*}
 Define $G':=Gh$. By right invariance of $\mu$ (on compact G $\mu$ is bi-invariant),
 \begin{equation*}
     G'\sim\mu, \qquad G' \perp X,
 \end{equation*}
 So $(G',X) \stackrel{d}{=} (G, X)$. Let $F(x, g):=\pi_{gx}(g \cdot \theta)$. Then 
 \begin{equation*}
     \pi_{hX}^{(\mu)}(h\cdot\theta)=F(X, G') \qquad \pi_{X}^{(\mu)}(\theta)=F(X, G), 
 \end{equation*}
 and hence $F(X,G')\stackrel{d}{=}F(X,G)$, that is
 \begin{equation*}
     \pi_{hX}^{(\mu)}(h\cdot\theta)\stackrel{d}{=} \pi_X^{(\mu)}(\theta).
 \end{equation*}

(iii) \emph{Orbit-average slice risk.} We have for each realization $(X,G)=(x,g)$ 
\begin{equation*}
    A^{(\mu)}_\alpha(x)=\left\{\vartheta: \pi_{gx}(g\cdot x) \geq \alpha \right\} = g^{-1}A_\alpha(g \cdot x).
\end{equation*} Hence $A^{(\mu)}_\alpha(X)=G^{-1}\cdot A_\alpha(G\cdot X)$ almost surely. Using the invariance of $\Gamma_\theta$ and the model we find:
\begin{align*}
    r_\alpha(\theta, \pi^{(\mu)})&=E_\theta \bigl[\sup_{\vartheta \in A^{(\mu)}_\alpha(X)} \Gamma_\theta(\vartheta) \bigr]= \int_G E_\theta \bigl[\sup_{\eta \in A_\alpha(g\cdot X)} \Gamma_\theta(g^{-1}\cdot\eta) \bigr]\mu(dg) \\
    &=\int_GE_{g\theta}\bigl[\sup_{\eta \in A_\alpha(X)} \Gamma_{g\theta}(\eta) \bigr]\mu(dg)=\int_Gr_\alpha(g\cdot \theta, \pi)\mu(dg).
\end{align*}
Taking $\sup_\theta$ and integrating over $\alpha$ then gives the inequality.
\end{proof}

\begin{remark}\label{rmk:amenable}
    For \textbf{amenable non-compact} $G$: For non‑compact amenable groups one replaces $\mu$ by a Følner sequence of probability measures that are asymptotically left‑ (or right‑) invariant, and defines the symmetrization consistently with that orientation. The validity and minimax arguments above then apply to each member of the sequence; the standard Hunt–Stein limiting argument (see \citet{Eaton1989} or \citet{lehmann2005testing}) yields an exactly invariant randomized procedure with the same minimax bound. There are some minor bookkeeping details about left or right invariance that need to be consistent throughout however. In the compact case we have treated above, we have been lax on that issue, since then any Haar is both left and right-invariant. Since section \ref{sec:equi} is written using left-invariance, it might be convenient to change the randomized symmetrization to using $G^{-1}$ instead of $G$, but all proofs will go through in a similar manner.

\end{remark}

\begin{lemma}\label{lem:A3}
Let $G$ act measurably on the sample space $\mathcal X$ and transitively on the parameter space
$\Theta$, and suppose the model $\{P_\theta:\theta\in\Theta\}$ is $G$-invariant. Let
$\hat\theta:\mathcal X\to\Theta$ be a measurable $G$-equivariant center, let
$M:\mathcal X\to\mathcal M$ be a measurable maximal $G$-invariant, and let
$\rho:\Theta\times\Theta\to[0,\infty)$ be a measurable $G$-invariant distance-like function such that
$\rho(\theta,\theta)=0$ for all $\theta\in\Theta$. Fix a reference point $\theta_0\in\Theta$, and let
\[
\mu_M := P_{\theta_0}\circ M^{-1}
\]
denote the law of $M(X)$ under $P_{\theta_0}$. Assume that under $P_{\theta_0}$ a regular conditional
distribution of $\rho(\theta_0,\hat\theta(X))$ given $M(X)$ exists, and write
\[
H(m,t)
:=
P_{\theta_0}\!\left\{
\rho(\theta_0,\hat\theta(X))\le t \,\middle|\, M(X)=m
\right\},
\qquad m\in\mathcal M,\ t\ge 0.
\]
Assume moreover that $m\mapsto H(m,t)$ is measurable for each fixed $t$, and that for
$\mu_M$-a.e.\ $m$, the map $t\mapsto H(m,t)$ is nondecreasing, right-continuous, and satisfies
$\lim_{t\to\infty}H(m,t)=1$.

Define the minimal feasible radius
\[
\tau_\alpha^*(m)
:=
\inf\{t\ge 0:H(m,t)\ge 1-\alpha\},
\qquad \alpha\in(0,1),
\]
and the associated $\rho$-ball rule
\[
A_\alpha^*(x)
:=
\{\vartheta\in\Theta:\rho(\vartheta,\hat\theta(x))\le \tau_\alpha^*(M(x))\}.
\]

Then:
\begin{enumerate}
\item[(1)] $m\mapsto \tau_\alpha^*(m)$ is measurable and finite $\mu_M$-a.s.

\item[(2)] $A_\alpha^*$ is $G$-equivariant and $M$-conditionally valid, i.e.
\[
A_\alpha^*(g\cdot x)=g\cdot A_\alpha^*(x)\quad \text{for all $g\in G$}
,\]
 and
\[
P_\theta\{\theta\in A_\alpha^*(X)\mid M(X)\}\ge 1-\alpha
\qquad \quad \text{for every $\theta\in\Theta$, $P_\theta$-a.s.}
\]
 In particular,
\[
P_\theta\{\theta\in A_\alpha^*(X)\}\ge 1-\alpha
\qquad \forall \theta\in\Theta.
\]

\item[(3)] For each $m$, $\alpha\mapsto \tau_\alpha^*(m)$ is nonincreasing. Hence
$\{A_\alpha^*\}_{\alpha\in(0,1)}$ is nested in $\alpha$.

\item[(4)] (Orbitwise minimality.) If $\tau:\mathcal M\to[0,\infty]$ is measurable and satisfies
\[
H(m,\tau(m))\ge 1-\alpha
\qquad \text{for }\mu_M\text{-a.e. }m,
\]
then
\[
\tau(m)\ge \tau_\alpha^*(m)
\qquad \text{for }\mu_M\text{-a.e. }m.
\]
Equivalently, among measurable $M$-dependent $\rho$-ball rules satisfying the conditional
coverage constraint, $A_\alpha^*$ has $\mu_M$-a.e.\ minimal radius.
\end{enumerate}
\end{lemma}

\begin{proof}
Let $M_0 \subseteq M$ be a measurable set with $\mu_M(M_0)=1$ such that, for every
$m \in M_0$, the map $t \mapsto H(m,t)$ is nondecreasing, right-continuous, and satisfies
$\lim_{t\to\infty} H(m,t)=1$. Redefine $H$ on $M_0^c$ by
\[
H(m,t)=1, \qquad m\in M_0^c,\ t\ge 0,
\]
and also set $H(m,\infty)=1$ for all $m$. This changes $H$ (and hence $\tau_\alpha^*$)
only on a $\mu_M$-null set, so it does not affect any statement of the lemma. We continue
to write $H$ and $\tau_\alpha^*$ for these modified versions.

For (1), for any $q>0$,
\[
\{m : \tau_\alpha^*(m) < q\}
=
\bigcup_{\substack{r\in\mathbb Q\\0\le r<q}}
\{m : H(m,r)\ge 1-\alpha\}.
\]
Indeed, if $\tau_\alpha^*(m)<q$, then by definition of the infimum there exists $s<q$ such
that $H(m,s)\ge 1-\alpha$. Choosing $r\in\mathbb Q$ with $s\le r<q$ and using monotonicity
of $H(m,\cdot)$ gives $H(m,r)\ge 1-\alpha$. The reverse inclusion is immediate. Since
$m\mapsto H(m,r)$ is measurable for each fixed $r$, the right-hand side is measurable.
Hence $m\mapsto \tau_\alpha^*(m)$ is measurable.

Since $1-\alpha<1$ and $H(m,t)\to 1$ as $t\to\infty$, there exists a finite $t=t(m)$ such that
$H(m,t)\ge 1-\alpha$ for every $m$ under the modified version. Therefore $\tau_\alpha^*(m)<\infty$
for every $m$ under that version, and in particular $\tau_\alpha^*<\infty$ $\mu_M$-a.s.
in the original statement.

We shall also use that
\[
H(m,\tau_\alpha^*(m)) \ge 1-\alpha
\qquad \text{for all } m\in M.
\]
Fix $m$. By definition of the infimum, for each $n\ge 1$ there exists $s_n<\tau_\alpha^*(m)+n^{-1}$
such that $H(m,s_n)\ge 1-\alpha$. By monotonicity,
\[
H\bigl(m,\tau_\alpha^*(m)+n^{-1}\bigr)\ge 1-\alpha.
\]
Letting $n\to\infty$ and using right-continuity of $H(m,\cdot)$ gives the claim.

For (2), equivariance is immediate from the assumptions on $\hat\theta$, $M$, and $\rho$.
For any $g\in G$ and $\vartheta\in\Theta$,
\begin{align*}
\vartheta \in A_\alpha^*(g\cdot x)
&\iff \rho\bigl(\vartheta,\hat\theta(g\cdot x)\bigr)
    \le \tau_\alpha^*\bigl(M(g\cdot x)\bigr) \\
&\iff \rho\bigl(\vartheta,g\cdot \hat\theta(x)\bigr)
    \le \tau_\alpha^*\bigl(M(x)\bigr) \\
&\iff \rho\bigl(g^{-1}\cdot \vartheta,\hat\theta(x)\bigr)
    \le \tau_\alpha^*\bigl(M(x)\bigr) \\
&\iff g^{-1}\cdot \vartheta \in A_\alpha^*(x) \\
&\iff \vartheta \in g\cdot A_\alpha^*(x).
\end{align*}
Thus
\[
A_\alpha^*(g\cdot x)=g\cdot A_\alpha^*(x).
\]

To prove $M$-conditional validity, fix $\theta\in\Theta$ and choose $g\in G$ such that
$\theta=g\cdot\theta_0$. Set
\[
Y := g^{-1}\cdot X.
\]
Then under $P_\theta$, we have $Y\sim P_{\theta_0}$ by $G$-invariance of the model.
Moreover, by $G$-equivariance of $\hat\theta$ and $G$-invariance of $M$ and $\rho$,
\[
\rho\bigl(\theta,\hat\theta(X)\bigr)
=
\rho\bigl(g\cdot\theta_0,\hat\theta(X)\bigr)
=
\rho\bigl(\theta_0,\hat\theta(Y)\bigr),
\qquad
M(X)=M(Y),
\qquad
P_\theta\text{-a.s.}
\]
Hence, for every bounded measurable $\phi:M\to\mathbb R$ and every $t\ge 0$,
\begin{align*}
E_\theta\!\left[
  1\bigl\{\rho(\theta,\hat\theta(X))\le t\bigr\}\phi(M(X))
\right]
&=
E_\theta\!\left[
  1\bigl\{\rho(\theta_0,\hat\theta(Y))\le t\bigr\}\phi(M(Y))
\right] \\
&=
E_{\theta_0}\!\left[
  1\bigl\{\rho(\theta_0,\hat\theta(X))\le t\bigr\}\phi(M(X))
\right] \\
&=
E_{\theta_0}\!\left[H(M(X),t)\phi(M(X))\right],
\end{align*}
where the last step is the defining property of the regular conditional distribution. Also,
for every bounded measurable $\phi$,
\[
E_\theta[\phi(M(X))]
=
E_\theta[\phi(M(Y))]
=
E_{\theta_0}[\phi(M(X))].
\]
Thus the law of $M(X)$ under $P_\theta$ is $\mu_M$, and therefore
\[
E_{\theta_0}[H(M(X),t)\phi(M(X))]
=
E_\theta[H(M(X),t)\phi(M(X))].
\]
Combining the last two displays shows that, for every fixed $t\ge 0$,
\[
P_\theta\bigl\{\rho(\theta,\hat\theta(X))\le t \mid M(X)\bigr\}
=
H(M(X),t),
\qquad
P_\theta\text{-a.s.}
\]

Let
\[
Z_\theta(X):=\rho(\theta,\hat\theta(X)),
\]
and let $K(m,\cdot)$ be the probability kernel on $[0,\infty]$ whose distribution function
is $t\mapsto H(m,t)$, i.e.
\[
K(m,[0,t])=H(m,t), \qquad t\ge 0.
\]
Since $m\mapsto H(m,t)$ is measurable for each fixed $t$, $K$ is a measurable kernel. By
the previous display and the monotone class theorem, $K(M(X),\cdot)$ is a version of the
conditional distribution of $Z_\theta(X)$ given $M(X)$ under $P_\theta$. Hence, by a standard
property of regular conditional distributions, for every measurable $a:M\to[0,\infty]$,
\[
P_\theta\{Z_\theta(X)\le a(M(X)) \mid M(X)\}
=
K(M(X),[0,a(M(X))])
=
H(M(X),a(M(X)))
\qquad
P_\theta\text{-a.s.}
\]
Applying this with $a=\tau_\alpha^*$ gives
\begin{align*}
    P_\theta\{\theta\in A_\alpha^*(X)\mid M(X)\}
&=
P_\theta\{Z_\theta(X)\le \tau_\alpha^*(M(X))\mid M(X)\}\\
&=
H(M(X),\tau_\alpha^*(M(X)))
\ge 1-\alpha,
\qquad
P_\theta\text{-a.s.}
\end{align*}
where the last inequality follows from the right-continuity argument above. Taking
expectations and using the tower property gives
\[
P_\theta\{\theta\in A_\alpha^*(X)\}\ge 1-\alpha,
\qquad \theta\in\Theta.
\]

For (3), if $\alpha_1<\alpha_2$, then $1-\alpha_1>1-\alpha_2$, so
\[
\{t\ge 0 : H(m,t)\ge 1-\alpha_1\}
\subseteq
\{t\ge 0 : H(m,t)\ge 1-\alpha_2\}.
\]
Therefore
\[
\tau_{\alpha_1}^*(m)\ge \tau_{\alpha_2}^*(m),
\]
so $\alpha\mapsto \tau_\alpha^*(m)$ is nonincreasing for each $m$. Consequently, if
$\alpha_1<\alpha_2$, then
\[
A_{\alpha_2}^*(x)
=
\{\vartheta : \rho(\vartheta,\hat\theta(x))\le \tau_{\alpha_2}^*(M(x))\}
\subseteq
\{\vartheta : \rho(\vartheta,\hat\theta(x))\le \tau_{\alpha_1}^*(M(x))\}
=
A_{\alpha_1}^*(x),
\]
so $\{A_\alpha^*\}_{\alpha\in(0,1)}$ is nested in $\alpha$.

For (4), let $\tau:M\to[0,\infty]$ be measurable and suppose
\[
H(m,\tau(m))\ge 1-\alpha
\qquad \text{for }\mu_M\text{-a.e. } m,
\]
with the convention $H(m,\infty)=1$. If $\tau(m)<\infty$, then $\tau(m)$ belongs to the
feasible set
\[
\{t\ge 0 : H(m,t)\ge 1-\alpha\},
\]
hence
\[
\tau(m)\ge \inf\{t\ge 0 : H(m,t)\ge 1-\alpha\} = \tau_\alpha^*(m).
\]
If $\tau(m)=\infty$, the same inequality is trivial. Therefore
\[
\tau(m)\ge \tau_\alpha^*(m)
\qquad \text{for }\mu_M\text{-a.e. } m.
\]
This is exactly the claimed orbitwise minimality. The equivalent formulation for measurable
$M$-dependent $\rho$-ball rules follows from the same conditioning argument as above: a
rule with radius $\tau(M(X))$ is $M$-conditionally valid if and only if
$H(m,\tau(m))\ge 1-\alpha$ for $\mu_M$-a.e. $m$.
\end{proof}

\begin{proposition}
\label{prop:gaussian-shift-balls}
Consider the Gaussian location experiment
\[
Z \sim N_d(u,I_d), \qquad u\in\mathbb{R}^d,
\]
and fix $\alpha\in(0,1)$. Let $C_\alpha^{\mathrm{tr}}$ be the class of measurable
translation-equivariant $(1-\alpha)$-confidence sets $C_\alpha$, i.e.
\[
C_\alpha(z+t)=C_\alpha(z)+t, \qquad z,t\in\mathbb{R}^d,
\]
such that
\[
P_u\{u\in C_\alpha(Z)\}\ge 1-\alpha, \qquad u\in\mathbb{R}^d.
\]
Let $h:[0,\infty)\to[0,\infty)$ be continuous and nondecreasing, and assume
\[
h(r)\to\infty \quad (r\to\infty),
\qquad
E\,h(r+\|W\|)<\infty \quad \forall r<\infty,
\]
where $W\sim N_d(0,I_d)$. Define
\[
r_\alpha(u,C_\alpha):=
E_u\Big[\sup_{v\in C_\alpha(Z)} h(\|v-u\|)\Big].
\]
Let $F_d$ be the distribution function of $\chi_d^2$, and let
\[
\rho_\alpha^2:=F_d^{-1}(1-\alpha).
\]
Then the ball rule
\[
C_\alpha^\star(z):=\{v\in\mathbb{R}^d:\|v-z\|\le \rho_\alpha\}
\]
minimizes $r_\alpha(u,\cdot)$ over $C_\alpha^{\mathrm{tr}}$ for every
$u\in\mathbb{R}^d$. In particular,
\[
\inf_{C_\alpha\in C_\alpha^{\mathrm{tr}}} r_\alpha(u,C_\alpha)
=
r_\alpha(u,C_\alpha^\star),
\qquad
u\in\mathbb{R}^d,
\]
and the right-hand side does not depend on $u$.
\end{proposition}

\begin{proof}
Fix $C_\alpha\in C_\alpha^{\mathrm{tr}}$. By translation equivariance there exists
a measurable set
\[
A:=-C_\alpha(0)\subset\mathbb{R}^d
\]
such that
\[
C_\alpha(z)=z-A, \qquad z\in\mathbb{R}^d.
\]
Let
\[
W:=Z-u\sim N_d(0,I_d).
\]
Then
\[
u\in C_\alpha(Z)
\iff
u\in Z-A
\iff
Z-u\in A
\iff
W\in A.
\]
Hence the coverage constraint is
\[
\gamma_d(A)\ge 1-\alpha,
\]
where $\gamma_d$ denotes standard Gaussian measure on $\mathbb{R}^d$.

Also,
\[
\sup_{v\in C_\alpha(Z)}\|v-u\|
=
\sup_{a\in A}\|Z-a-u\|
=
\sup_{a\in A}\|W-a\|.
\]
Define
\[
\varphi_A(w):=\sup_{a\in A}\|w-a\|, \qquad w\in\mathbb{R}^d.
\]
Then
\[
r_\alpha(u,C_\alpha)=E\,h(\varphi_A(W)),
\]
which no longer depends on $u$. Thus it suffices to minimize
\[
E\,h(\varphi_A(W))
\]
over measurable $A\subset\mathbb{R}^d$ subject to $\gamma_d(A)\ge 1-\alpha$.

If $A$ were unbounded, then for every $w\in\mathbb{R}^d$ and every $M>0$ there
would exist $a\in A$ with $\|w-a\|>M$, so $\varphi_A(w)\ge M$. Since $h$ is
nondecreasing and $h(r)\to\infty$, this would force $h(\varphi_A(w))=\infty$
for every $w$, hence $E\,h(\varphi_A(W))=\infty$. Therefore every finite-risk
competitor is bounded.

Now replace $A$ by
\[
K:=\overline{\operatorname{conv}}(A).
\]
Then $\gamma_d(K)\ge \gamma_d(A)$, since $A\subseteq K$. Moreover, for each
fixed $w$, the map $a\mapsto \|w-a\|$ is convex, so
\[
\sup_{a\in K}\|w-a\|=\sup_{a\in A}\|w-a\|,
\]
hence $\varphi_K=\varphi_A$. Therefore, replacing $A$ by $K$ if necessary, we
may assume from now on that $A$ is compact and convex (by convention confidence set are assumed closed, so the closure of the hull is then not needed for compactness). Since
$\gamma_d(A)\ge 1-\alpha>0$, $A$ has nonempty interior, so $A$ is a convex body.

For $\omega\in S^{d-1}$, let $\sigma_\omega$ denote reflection across
$\omega^\perp$, and define the Minkowski symmetral
\[
B_\omega A:=\frac12(A+\sigma_\omega A),
\]
where $+$ here denotes the Minkowski sum. Since $\gamma_d(\sigma_\omega A)=\gamma_d(A)$, Ehrhard's inequality \citep{Erhard-orig,BORELL2003663} gives
\[
\Phi^{-1}\bigl(\gamma_d(B_\omega A)\bigr)
\ge
\frac12\Phi^{-1}\bigl(\gamma_d(A)\bigr)
+
\frac12\Phi^{-1}\bigl(\gamma_d(\sigma_\omega A)\bigr)
=
\Phi^{-1}\bigl(\gamma_d(A)\bigr),
\]
hence
\[
\gamma_d(B_\omega A)\ge \gamma_d(A).
\]

To compare risks, define for $t\ge 0$
\[
E_A(t):=\{w\in\mathbb{R}^d:\varphi_A(w)\le t\}.
\]
If $B(w,t)$ denotes the closed Euclidean ball of radius $t$ centered at $w$,
then
\[
\varphi_A(w)\le t
\iff
A\subseteq B(w,t)
\iff
w\in \bigcap_{a\in A} B(a,t),
\]
so
\[
E_A(t)=\bigcap_{a\in A} B(a,t).
\]
In particular, $E_A(t)$ is convex.

We then claim that
\[
\frac12\bigl(E_A(t)+\sigma_\omega E_A(t)\bigr)\subseteq E_{B_\omega A}(t).
\]
Indeed, let $w_1\in E_A(t)$ and let $w_2\in \sigma_\omega E_A(t)$. Write
$w_2=\sigma_\omega u_2$ for some $u_2\in E_A(t)$, and set
\[
w:=\frac12(w_1+w_2)=\frac12(w_1+\sigma_\omega u_2).
\]
Take any $x\in B_\omega A$. Then
\[
x=\frac12(a+\sigma_\omega b)
\]
for some $a,b\in A$. Hence
\[
\|x-w\|
\le
\frac12\|a-w_1\|+\frac12\|\sigma_\omega b-\sigma_\omega u_2\|
=
\frac12\|a-w_1\|+\frac12\|b-u_2\|
\le t,
\]
because $w_1,u_2\in E_A(t)$. Therefore $w\in E_{B_\omega A}(t)$, proving the
claim.

Applying Ehrhard's inequality to the convex sets $E_A(t)$ and
$\sigma_\omega E_A(t)$ yields
\[
\gamma_d(E_{B_\omega A}(t))
\ge
\gamma_d\!\left(\frac12\bigl(E_A(t)+\sigma_\omega E_A(t)\bigr)\right)
\ge
\gamma_d(E_A(t)).
\]
Since
\[
\gamma_d(E_A(t))=P\{\varphi_A(W)\le t\},
\]
it follows that
\[
P\{\varphi_{B_\omega A}(W)\le t\}\ge P\{\varphi_A(W)\le t\}
\qquad
\forall t\ge 0.
\]
That is,
\[
\varphi_{B_\omega A}(W)\le_{\mathrm{st}} \varphi_A(W),
\]
and therefore, because $h$ is nondecreasing,
\[
E\,h(\varphi_{B_\omega A}(W))\le E\,h(\varphi_A(W)).
\]

Now choose a countable dense set $T\subset S^{d-1}$. By the universal-sequence
theorem for Minkowski symmetrization (see Coupier and Davydov
\citep{MinkowskiSymm}), there exist directions
$\omega_1,\omega_2,\dots\in T$ such that, if
\[
A_0:=A,
\qquad
A_{m+1}:=B_{\omega_{m+1}}A_m
\qquad (m\ge 0),
\]
then
\[
A_m \to B_R
\]
in Hausdorff distance for some centered Euclidean ball
\[
B_R:=\{x\in\mathbb{R}^d:\|x\|\le R\}.
\]

The one-step inequalities imply
\[
\gamma_d(A_m)\ge \gamma_d(A)
\quad\text{and}\quad
E\,h(\varphi_{A_m}(W))\le E\,h(\varphi_A(W))
\qquad (m\ge 0).
\]
Choose $R_0<\infty$ such that $A\subseteq B_{R_0}$. Since reflections preserve
$B_{R_0}$ and Minkowski symmetrization stays inside the same centered ball,
\[
A_m\subseteq B_{R_0}
\qquad \forall m.
\]
Hence
\[
\varphi_{A_m}(w)\le R_0+\|w\|.
\]
By assumption,
\[
E\,h(R_0+\|W\|)<\infty.
\]
Also,
\[
|\varphi_{A_m}(w)-\varphi_{B_R}(w)|
\le
d_H(A_m,B_R)\to 0.
\]
Therefore, by dominated convergence,
\[
E\,h(\varphi_{B_R}(W))
=
\lim_{m\to\infty} E\,h(\varphi_{A_m}(W))
\le
E\,h(\varphi_A(W)).
\]

Similarly, Hausdorff convergence implies
\[
1_{A_m}(w)\to 1_{B_R}(w)
\qquad
\text{for every } w\notin \partial B_R,
\]
and $\gamma_d(\partial B_R)=0$. Hence dominated convergence gives
\[
\gamma_d(B_R)
=
\lim_{m\to\infty}\gamma_d(A_m)
\ge
\gamma_d(A)
\ge
1-\alpha.
\]

Now let $\rho_\alpha$ be the unique radius such that
\[
\gamma_d(B_{\rho_\alpha})=1-\alpha.
\]
Equivalently,
\[
\rho_\alpha^2 = F_d^{-1}(1-\alpha).
\]
Since Gaussian measure of centered balls is increasing in the radius,
$\rho_\alpha\le R$. For a centered Euclidean ball,
\[
\varphi_{B_r}(w)=r+\|w\|,
\]
so by monotonicity of $h$,
\[
E\,h(\varphi_{B_{\rho_\alpha}}(W))
\le
E\,h(\varphi_{B_R}(W))
\le
E\,h(\varphi_A(W)).
\]

Translating back to confidence sets, and using $B_{\rho_\alpha}=-B_{\rho_\alpha}$,
the centered ball $B_{\rho_\alpha}$ corresponds to
\[
C_\alpha^\star(z)=z-B_{\rho_\alpha}
=
\{v\in\mathbb{R}^d:\|v-z\|\le \rho_\alpha\}.
\]
Therefore
\[
r_\alpha(u,C_\alpha^\star)\le r_\alpha(u,C_\alpha)
\qquad
\forall u\in\mathbb{R}^d.
\]
Since $C_\alpha$ was arbitrary, the proof is complete.
\end{proof}

\end{document}